\documentclass[12pt,reqno]{amsart}
\usepackage{latexsym}
\usepackage{amsmath,amssymb,amsfonts}
\usepackage[dvips]{graphicx,color}
\usepackage{mathrsfs}
\usepackage[abbrev]{amsrefs}

\usepackage{latexsym}
\usepackage{delarray} 

\newtheorem{thm}{Theorem}[section]
\newtheorem{cor}[thm]{Corollary}
\newtheorem{prop}[thm]{Proposition}
\newtheorem{lem}[thm]{Lemma}
\newtheorem{definition}[thm]{Definition}

\newtheorem{assum}[thm]{Assumption}
\newtheorem{rem}[thm]{Remark}

\numberwithin{equation}{section}
\numberwithin{thm}{section}

\makeatletter
\newcommand{\vertiii}[1]{{\left\vert\kern-0.25ex\left\vert\kern-0.25ex\left\vert #1 
    \right\vert\kern-0.25ex\right\vert\kern-0.25ex\right\vert}}

\title[Nonlinear damped wave equations on metric measure spaces]{
Global existence and asymptotic behavior for semilinear damped wave equations on measure spaces}
\author[M. Ikeda, K. Taniguchi and Y. Wakasugi]
{Masahiro Ikeda, Koichi Taniguchi and Yuta Wakasugi}

\address[M. Ikeda]
{Faculty of Science and Technology,
Keio University, 
3-14-1 Hiyoshi, Kohoku-ku, Yokohama, 223-8522, Japan/ Center for Advanced Intelligence Project
RIKEN, Japan.}
\email{masahiro.ikeda@keio.jp/masahiro.ikeda@riken.jp}
\address[K. Taniguchi]
{Advanced Institute for Materials Research,
Tohoku University,
2-1-1 Katahira, Aoba-ku, Sendai, 980-8577, Japan.}
\email{koichi.taniguchi.b7@tohoku.ac.jp}
\address[Y. Wakasugi]
{Laboratory of Mathematics, 
Graduate School of Advanced Science and Engineering, 
Hiroshima University,
Higashi-Hiroshima, 739-8527, Japan.
}
\email{wakasugi@hiroshima-u.ac.jp}


\begin{document}


%
\begin{abstract}
The purpose of this paper is to prove the small data global existence of solutions to the semilinear damped wave equation 
\[
\partial_t^2 u + Au + \partial_t u = |u|^{p-1}u
\]
on a measure space $X$ with a self-adjoint operator $A$ on $L^2(X)$. 
Under a certain decay estimate on the corresponding heat semigroup, we establish the linear estimates which generalize the so-called Matsumura estimates. 
Our approach is based on a direct spectral analysis analogous to the Fourier analysis.
The self-adjoint operators treated in this paper include some important examples such as the Laplace operators on Euclidean spaces, 
the Dirichlet Laplacian on an arbitrary open set,  the Robin Laplacian on an exterior domain, the Schr\"odinger operator, the elliptic operator, the Laplacian on Sierpinski gasket, and the fractional Laplacian.
\end{abstract}

\maketitle


\section{Introduction}
Let $(X,\mu)$ be a $\sigma$-finite measure space, and let $A$ be a non-negative and self-adjoint operator on $L^2(X)$. 
In this paper, we study the Cauchy problem of the semilinear damped wave equation
\begin{equation}\label{ndw}
\begin{cases}
\partial_t^2 u + Au + \partial_t u = F(u),
\quad &(t,x) \in (0,T) \times X,\\
u(0,x) = u_0(x),\quad \partial_t u(0,x) = u_1(x),\quad &x\in X,
\end{cases}
\end{equation}
where $T>0$, $u= u(t,x)$ is an unknown real-valued function on $(0,T) \times X$, and 
$u_i = u_i(x)$ are prescribed real-valued functions on $X$ for $i=0,1$. Here $F : \mathbb R \to \mathbb R$ satisfies $F(0)=0$, $F \in C^1(\mathbb R)$, and 
\begin{equation}\label{condi-F}
|F(z_1) - F(z_2)| \le C (|z_1|+|z_2|)^{p-1}|z_1-z_2|
\end{equation}
for any $z_1,z_2 \in\mathbb R$ and some $p>1$. The typical examples of $F(u)$ are
\[
F(u) = \pm |u|^p, \pm |u|^{p-1}u.
\]


The damped wave equation \eqref{ndw} with $X = \mathbb R^d$ and $A=-\Delta$, i.e.,
\begin{equation}\label{ndw-usual}
\partial_t^2 u -\Delta u + \partial_t u = F(u),
\quad (t,x) \in (0,T) \times \mathbb R^d,
\end{equation}
has been well studied. In the linear problem (i.e., $F(u)=0$), 
Matsumura \cite{Mat-1976} established $L^{q_1}$-$L^{q_2}$ estimates (later called the Matsumura estimates)
\begin{equation}\label{Matsumura}
\|u(t)\|_{L^{q_2}}
\lesssim 
\langle t \rangle^{-\frac{d}{2}(\frac{1}{q_1}-\frac{1}{q_2})}
(\|u_0\|_{L^{q_1}} + \|u_1\|_{L^{q_1}}) + e^{-\frac{t}{4}}
(\|u_0\|_{H^{[\frac{d}{2}] + 1}} + \|u_1\|_{H^{[\frac{d}{2}]}}) 
\end{equation}
for solutions $u=u(t)$ with initial data $(u_0,u_1)$, where $1 \le q_1 \le 2 \le q_2 \le \infty$ 
and $[d/2]$ denotes the integer part of $d/2$.
Then, nonlinear problems have been investigated based on the Matsumura estimates.

In the semilinear problem \eqref{ndw-usual} in $L^q$-framework with $1\le q\le 2$, the exponent 
\[
p=1+\frac{2q}{d}
\] 
is known to be an important critical number.
This exponent appears in terms of scale invariance of the semilinear heat equation with semilinear term $u^p$ in $L^q$-framework, 
and is critical in the sense that the global existence holds with small initial data in the case $p>1+2q/d$, but, in general, not in the case $1<p< 1+2q/d$. 
This phenomenon was first proved by Fujita \cite{Fujita} in $L^1$-framework, and so $p=1+2/d$ ($q=1$) is called the Fujita exponent. 
In the critical case $p=1+2q/d$, the asymptotic behavior of solutions depends on $q$: the global existence with small data when $q>1$, but not when $q=1$. 
These results have also been shown for the damped wave equation \eqref{ndw-usual}.
The $L^1$-framework was proved by Todorova and Yordanov \cite{TodYor-2001} (their result does not include the critical case $p=1+2/d$).
In the critical case  $p=1+2/d$, the global nonexistence with small initial data was proved by Zhang \cite{Zha-2001}. 
In the $L^q$-framework with $1<q\le 2$, the corresponding results were proved by 
Ikehata and Ohta \cite{IkeOht-2002}. 
For further developments and related works on \eqref{ndw-usual}, 
we refer to \cites{FIW-2020,HKN-2004,HO-2004,IIOW-2019,Ike-2002,Ike-2003,IkeTan-2005,Nar-2004,Nis-2003,Ono-2003,Pal-2020,Tak-2011}.
See Section 1 in \cite{IIOW-2019} for a summary of study of damped wave equations.

As mentioned above, the damped wave equation \eqref{ndw-usual}, i.e., the case $X = \mathbb R^d$ and $A=-\Delta$, has been well investigated. However, 
there are other interesting and important settings such as some perturbed problems, exterior problems, problems on manifolds, groups or measure spaces. 
These settings can be treated in a unified manner via the spectral approach, and several studies are known in this direction.
Ikehata and Nishihara \cite{IkeNis-2003} studied the asymptotic behavior of solutions to the linear abstract Cauchy problem
\begin{equation}\label{dw}
\begin{cases}
\partial_t^2 v + Av + \partial_t v = 0,\quad &(t,x) \in (0,\infty) \times X,\\
v(0,x) = v_0(x),\quad \partial_t v(0,x) = v_1(x),\quad &x\in X,
\end{cases}
\end{equation}
and Chill and Haraux \cite{ChiHar-2003} improved the convergence rates in the asymptotic behavior in \cite{IkeNis-2003}. 
The Matsumura type estimates for \eqref{dw} have been studied under some assumptions on $X$, $\mu$, and $A$ (see Radu, Todorova and Yordanov \cite{RTY-2011}, and Taylor and Todorova \cite{TayYor-2020}). 
The study of asymptotic behavior for \eqref{dw} has been further developed (see \cite{Nishiyama-2016,RTY-2016,Sob-2020}).
In terms of semilinear problems \eqref{ndw}, the global well-posedness was recently studied by 
Ruzhansky and Tokmagambetov \cite{RT-2019}, which treated self-adjoint operators $A$ with discrete spectrum.

The first purpose of this paper is to prove the Matsumura type estimates for solutions to
the linear problem \eqref{dw}.
The Matsumura type estimates have already beed studied by \cite{RTY-2011,TayYor-2020} in the abstract setting. However, 
our approach is different from these and 
is based on a direct argument analogous to the Fourier analysis; thereby, it is applicable to several evolution equations (see Remark~\ref{rem:linear} below).
The second purpose is to study the global existence for the semilinear problem \eqref{ndw} with small initial data.
Our result is an improvement and generalization of the result by \cite{IkeOht-2002} (see Remark~\ref{rem:compare} below).

In this paper, we use the spectral decomposition of $A$, instead of the Fourier transform, which plays a fundamental and important role when $X=\mathbb R^d$ and $A=-\Delta$. 
For a Borel measurable function $\phi$ on $\mathbb R$, 
an operator $\phi(A)$ is defined by 
\[
\phi(A) = \int_{\sigma(A)} \phi(\lambda)\, dE_{A}(\lambda)
\]
with the domain 
\[
\mathrm{Dom} (\phi(A)) = 
\left\{
f \in L^2(X) : \int_{\sigma(A)} |\phi(\lambda)|^2\, d\|E_A(\lambda) f\|_{L^2}^2 < \infty
\right\},
\]
where $\sigma(A)$ denotes the spectrum of $A$ and $\{E_{A}(\lambda)\}_ {\lambda\in\mathbb R}$ is the spectral resolution of the identity for $A$. 
We apply the spectral decomposition of $A$ to the linear problem \eqref{dw}, 
and obtain the representation formula
\begin{equation}\label{eq:formula}
v(t,x) = \mathcal D(t,A)(v_0+v_1) + \partial_t \mathcal D(t, A) v_0,
\end{equation}
where $\mathcal D(t,A)$ is defined by the spectral decomposition 
\[
\mathcal D(t,A) := \int_{\sigma(A)} \mathcal D(t,\lambda)\, dE_{A}(\lambda)
\]
with 
\[
\mathcal D(t,\lambda):=
\begin{cases}
\dfrac{e^{-\frac{t}{2}}\sinh \left(t\sqrt{\frac14 - \lambda} \right)}{\sqrt{\frac14 - \lambda}}\quad &\left(0\le \lambda < \dfrac14\right),\\
t e^{-\frac{t}{2}}\quad &\left(\lambda = \dfrac14\right),\\
\dfrac{e^{-\frac{t}{2}}\sin \left(t\sqrt{\lambda-\frac14}\right)}{\sqrt{\lambda-\frac14}}\quad &\left(\lambda > \dfrac14\right).
\end{cases}
\]

\medskip

In the last part of this section, let us introduce the notations used in this paper. 
For $a, b \ge0$, the symbols $a \lesssim b$ and $b \gtrsim a$ mean that there exists a constant $C>0$ such that $a \le C b$.
The symbol $a \sim b$ means that $a \lesssim b$ and $b \lesssim a$ happen simultaneously. 
We use the notation $\langle t \rangle := (1 + t^2)^{\frac12}$. We denote by $\mathbf{1}_{I}=\mathbf{1}_{I}(\lambda)$ the characteristic function of the interval $I \subset \mathbb R$. 
The space $L^q(X)$ is a space of 
all $\mu$-measurable functions on $X$ for which the $q$-th power of the absolute value is Lebesgue integrable, and
$\|\cdot \|_{L^q}$ stands for the usual $L^q(X)$-norm. 
For $s\ge0$, 
the Sobolev space $H^s(A)$ is defined by 
\[
H^s(A) := \{ f\in L^2(X) : \|f\|_{H^s(A)}<\infty\}
\]
with the norm
\[
\|f\|_{H^s(A)} := \|(I+A)^{\frac{s}{2}} f\|_{L^2(X)}.
\]
The space $H^{-s}(A)$ denotes the dual space of $H^s(A)$. 
We use the notation $\|\cdot\|_{V \to W}$  for the operator norm from a normed space $V$ to a normed space $W$.


\section{Statement of the results}

In the study of damped wave equation \eqref{ndw-usual}, the well-known decay estimate
\[
\|e^{t\Delta}\|_{L^2\to L^\infty} \lesssim t^{-\frac{d}{4}},\quad t>0
\]
is a fundamental tool. Taking this into account, throughout this paper, we make the following assumption for the study of the damped wave equations \eqref{ndw} in the abstract setting.

\begin{assum}\label{assum:A}
$A$ is a non-negative self-adjoint operator on $ L^2(X)$ and 
its semigroup $(e^{-tA})_{t\ge0}$ satisfies 
\begin{equation}\label{ass:A}
\|e^{-tA}\|_{L^2\to L^\infty} \lesssim t^{-\alpha},\quad t>0
\end{equation}
for some $\alpha>0$. 
\end{assum}

There are many examples of operators satisfying Assumption \ref{assum:A}: 
Dirichlet Laplacian on an open set, 
Robin Laplacian on an exterior domain,  
Schr\"odinger operator with a Kato type potential
or a Dirac delta potential, 
elliptic operator, 
Laplace operator on Sierpinski gasket,
fractional Laplacian, 
Laplace-Beltrami operator on a Riemann manifold, Laplace operator on a homogeneous group, and sub-Laplacian operator on Heisenberg group.  
See Section \ref{sec:6} for the details.

\subsection{Linear abstract Cauchy problems}

Our first result is an abstract version of 
the Matsumura estimates and asymptotic behavior of solutions to \eqref{dw}. 

\begin{thm}\label{thm:Matsumura-est}
Suppose that $A$ satisfies Assumption \ref{assum:A}. 
Let $k=0,1$, $q\in [1,2]$, $s\ge0$ and $\beta > 2\alpha + k + s-1$. 
Then: 
\begin{itemize}
\item[(i)] {\rm (Matsumura type estimates)}
\begin{equation}\label{Mat-est1}
\|\partial_t^k A^{\frac{s}{2}} \mathcal D(t,A) f\|_{L^\infty}
\lesssim \langle t\rangle^{-\frac{2\alpha}{q} - k - \frac{s}{2}} (\|f\|_{L^q} + e^{-\frac{t}{4}}\|f\|_{H^\beta(A)})
\end{equation}
for any $f\in L^q(X) \cap H^\beta(A)$ and $t>0$, and 
\begin{equation}\label{Mat-est2}
\|\partial_t^k A^{\frac{s}{2}} \mathcal D(t,A) f\|_{L^2}
\lesssim \langle t\rangle^{-2\alpha(\frac{1}{q} -\frac12) - k - \frac{s}{2}} (\|f\|_{L^q} + e^{-\frac{t}{4}}\|f\|_{H^{k+s-1}(A)})
\end{equation}
for any $f\in L^q(X) \cap H^{k+s-1}(A)$ and $t>0$.
\item[(ii)] {\rm (Asymptotic behavior)} 
\begin{equation}\label{est:q-infty-asy}
\|\partial_t^k A^{\frac{s}{2}} (\mathcal D(t,A) - e^{-tA})f\|_{L^\infty}
\lesssim \langle t\rangle^{-\frac{2\alpha}{q} - k - \frac{s}{2}-1} (\|f\|_{L^q} + e^{-\frac{t}{4}}\|f\|_{H^\beta(A)})
\end{equation}
for any $f\in L^q(X) \cap H^\beta(A)$ and $t\ge1$, and 
\begin{equation}\label{est:q-2-asy}
\|\partial_t^k A^{\frac{s}{2}} (\mathcal D(t,A) - e^{-tA}) f\|_{L^2}
\lesssim \langle t\rangle^{-2\alpha(\frac{1}{q} -\frac12) - k - \frac{s}{2}-1} 
(\|f\|_{L^q} + e^{-\frac{t}{4}}\|f\|_{H^{k+s-1}(A)})
\end{equation}
for any $f\in L^q(X) \cap H^{k+s-1}(A)$ and $t\ge1$.
\end{itemize}
\end{thm}

\begin{rem}\label{rem:linear}
The Matsumura type estimates \eqref{Mat-est1} and \eqref{Mat-est2} in Theorem \ref{thm:Matsumura-est} have already been obtained under Assumption \ref{assum:A} with Markov property in \cite{RTY-2011}, 
where the proof is based on the diffusion phenomenon in Hilbert spaces, that is, 
the asymptotic behavior of solutions of \eqref{dw} to $e^{-tA}(v_0 + v_1)$. 
On the other hand, our proofs of \eqref{Mat-est1} and \eqref{Mat-est2} are different from it and is based on the direct argument analogous to the Fourier analysis without using the diffusion phenomenon. 
\end{rem}

\subsection{Semilinear abstract Cauchy problems}

Our second result is about the global existence for the semilinear problem \eqref{ndw}. 
Let us introduce the notion of solutions used in this paper. 

\begin{definition}
Let $T >0$ and $(u_0, u_1) \in H^1(A) \times L^2(X)$. 
A function $u : [0,T] \times X \to \mathbb R$ is called a mild solution to \eqref{ndw} with initial data $(u(0), \partial_t u(0)) = (u_0, u_1)$ if 
$u \in C([0,T] ; H^1(A)) \cap C^1([0,T] ; L^2(X))$ satisfies the integral equation 
\[
u(t) = u_L(t) + \int_0^t \mathcal D(t-\tau, A)(F(u(\tau)))\, d\tau
\]
in $C([0,T] ; H^1(A)) \cap C^1([0,T] ; L^2(X))$, 
where $u_L=u_L(t)$ is the solution to the linear problem \eqref{dw} with initial data $(u_L(0), \partial_t u_L(0)) = (u_0,u_1)$,
that is,
\[
	u_{L}(t) = \mathcal D(t,A)(u_0+u_1) + \partial_t \mathcal D(t, A) u_0.
\]
\end{definition}

To state our result, we introduce the important exponent
\[
p_F = p_F(\alpha,q) := 1 + \frac{q}{2\alpha}
\] 
for $\alpha>0$ and $q\in [1,2]$. The case $\alpha = d/4$ and $q=1$ is the Fujita exponent.
We note that
$p > p_F(\alpha, q)$
is equivalent with
$q < 2\alpha (p-1)$.
As we will see later,
$p_F(\alpha, q)$
is in some sense the critical exponent
for $L^q$-initial data.

We study the global existence for  \eqref{ndw}
under the condition
\begin{align}\label{supercritical}
    p> p_F(\alpha,1),\ 
    \frac{2}{p} < 2\alpha (p-1),
    \ 
    \text{and}
    \ 
    q \in \left[
    \max\left\{1,\frac{2}{p} \right\}, \min\{2,2\alpha (p-1)\}
    \right].
\end{align}
\medskip

In the semilinear problem, in addition to Assumption \ref{assum:A}, we make the following assumption if $\alpha>1/2$.

\begin{assum}\label{assum:B}
If $\alpha > 1/2$, then 
there exist
$q_1 \in [1,2)$ and $q_2 \in (2,4\alpha)$
such that the operator $A$ satisfies 
\[
\|e^{-tA}\|_{L^{q_j}\to L^{q_j}} \lesssim 1
\quad \text{and}\quad 
\|e^{-tA}\|_{L^{q_j}\to L^{\infty}} \lesssim t^{-\frac{2\alpha}{q_j}},\quad t>0
\]
for $j=1,2$.

\end{assum}

\begin{rem}\label{rem:assum:B}
Assumption \ref{assum:B} is actually needed only to prove the critical Sobolev inequality \eqref{c-Sobolev} in Lemma~\ref{lem:cS} below. 
To simplify the description, we always impose this assumption in the semilinear problem (see also Remark \ref{rem:additional-assum} below).
\end{rem}

We have the result on 
the local (in time) existence of mild solutions to the problem \eqref{ndw}.

\begin{prop}\label{prop:local}
Suppose that $A$ satisfies Assumptions \ref{assum:A} and \ref{assum:B}.
Let $p\in (1,\infty)$ and $(u_0, u_1) \in H^1(A) \times L^2(X)$. 
Assume that 
\begin{equation}\label{additional-assum}
p \le \frac{2\alpha}{2\alpha-1}\quad \text{if }\alpha > \frac12.
\end{equation}
Then there exist a time $T>0$,
depending only on
$\|u_0\|_{H^1(A)}$ and $\|u_1\|_{L^2}$, and 
a unique mild solution $u \in C([0,T] ; H^1(A)) \cap C^1([0,T] ; L^2(X))$ to \eqref{ndw} with initial data $(u_0, u_1)$. 
Moreover, if $T_{\max} < \infty$, then 
\[
\lim_{t\to T_{\max}} \left(
\|u(t)\|_{H^1(A)} + \|u_t(t)\|_{L^2}
\right) = \infty,
\]
where
$T_{\max}=T_{\max}(u_0, u_1)$ denotes the maximal existence time of $u$ defined by 
\[
\begin{split}
T_{\max}  := \sup 
\Big\{ T>0 : &
\text{ There exists a unique mild solution $u$ to \eqref{ndw}}\\
& \text{in $C([0,T] ; H^1(A)) \cap C^1([0,T] ; L^2(X))$ with initial data $(u_0, u_1)$}
\Big\}.
\end{split}
\]
\end{prop}

Our main result is the following:

\begin{thm}\label{thm:sdge}
Suppose that $A$ satisfies Assumptions \ref{assum:A} and \ref{assum:B}.
Let $q \in [1,2]$ and $p\in (1,\infty)$ satisfy
\eqref{supercritical}
and \eqref{additional-assum}.
Then, there exists a constant $\epsilon_0>0$, depending on $X$, $\alpha$, $p$ and $q$, such that
if initial data $(u_0, u_1) \in (H^1(A) \cap L^q(X)) \times (L^2(X) \cap L^q(X))$ satisfies 
\[
I_{0} = I_{0}(u_0,u_1) := \|u_0\|_{L^q} + \|u_0\|_{H^1(A)} +  \|u_1\|_{L^q} + \|u_1\|_{L^2} \le \epsilon_0,
\]
there exists a unique global (in time) mild solution $u \in C([0,\infty) ; H^1(A)) \cap C^1([0,\infty) ; L^2(X))$ to the problem \eqref{ndw} with $(u(0), u_t(0))=(u_0, u_1)$.
Furthermore, the solution $u=u(t)$ satisfies 
\begin{equation}\label{global-est}
\begin{split}
& \langle t\rangle^{2\alpha (\frac{1}{q} - \frac12) + 1} (\log (2+t))^{-\delta} \|u_t(t)\|_{L^2}\\
& \hspace{3cm}+
\langle t\rangle^{2\alpha (\frac{1}{q} - \frac12) + \frac12}
 \|A^\frac12 u(t)\|_{L^2}
+
\langle t\rangle^{2\alpha (\frac{1}{q} - \frac12)}
\|u(t)\|_{L^2}
\lesssim 
I_0 
\end{split}
\end{equation}
for any $t>0$,
where
$\delta = 0$ if $q < 2\alpha (p-1)$
and $\delta = 1$ if $q = 2\alpha (p-1)$.
\end{thm}

\begin{rem}\label{rem:compare}
In the case of
$X = \mathbb{R}^d$,
Theorem \ref{thm:sdge} slightly improves
Theorems~1.1 and 1.2 in \cite{IkeOht-2002}.
First, Theorem \ref{thm:sdge}
includes the critical case
$q = 2\alpha(p-1)$,
i.e.,
$p=p_F(\alpha,q)$
except for $q = 1$.
Moreover, the lower bound
\[
\frac{\sqrt{d^2 + 16 d}- d}{4} \le q,
\]
is relaxed in Theorem \ref{thm:sdge} 
(more detailed comparison will be given in
Corollary \ref{cor:sdge1} below). 
\end{rem}

\begin{rem}
The conditions \eqref{supercritical} and \eqref{additional-assum} in Theorem \ref{global-est} imply that 
$\alpha < 1 + 1/\sqrt{2}$.
\end{rem}

\begin{rem}\label{rem:additional-assum}
Assumption \ref{assum:B} is needed only in the case $p = 2\alpha/(2\alpha-1)$ in Theorem~\ref{thm:sdge} in order to use the critical Sobolev inequality \eqref{c-Sobolev} in Lemma~\ref{lem:cS} below (see also Remark~\ref{rem:assum:B}). 
\end{rem}

\begin{rem}
The exponent $p_F(\alpha,q)$ is known to be the threshold of global in time existence or nonexistence for \eqref{ndw} in some special cases, for example, the cases of cone domains in $\mathbb R^d$, the Heisenberg group and compact Lie groups (see e.g. \cites{IS-2019,GP-2020,Pal-2021}).
\end{rem}


\section{Proof of Theorem \ref{thm:Matsumura-est}}

Throughout this section, we suppose that $A$ satisfies Assumption \ref{assum:A}. 
To prove Theorem \ref{thm:Matsumura-est}, 
let us prepare three lemmas.

\begin{lem}\label{lem:LpLq}
Let $k =0,1$, $s\ge0$ and $1\le q_1 \le 2 \le q_2 \le \infty$. Then 
\begin{equation}\label{LpLq_1}
\|\partial_t^k A^{\frac{s}{2}}e^{-tA}\|_{L^{q_1}\to L^{q_2}} \lesssim t^{-2\alpha(\frac{1}{q_1}-\frac{1}{q_2})-k -\frac{s}{2}},
\end{equation}
\begin{equation}\label{LpLq_2}
\|\partial_t^k A^{\frac{s}{2}}\chi (A) e^{-tA}\|_{L^{q_1}\to L^{q_2}} \lesssim
\langle t\rangle^{-2\alpha(\frac{1}{q_1}-\frac{1}{q_2})-k -\frac{s}{2}}
\end{equation}
for any $t>0$, 
where $\chi \in C^\infty_0(\mathbb R)$. 
\end{lem}
\begin{proof}
In order to prove \eqref{LpLq_1}, 
it is enough to show the four cases 
\begin{equation}\label{eq:q1q2}
(q_1,q_2) = (2,2), (2,\infty), (1,2), (1,\infty)
\end{equation}
and apply the Riesz-Thorin interpolation theorem.
Since $A$ is non-negative and $\partial_t^k e^{-tA} = (-A)^k e^{-tA}$, we have 
\[
\|\partial_t^k A^{\frac{s}{2}}e^{-tA}\|_{L^2\to L^2} \lesssim t^{-k-\frac{s}{2}},
\]
and 
\begin{equation}\label{lem:3.1_1}
\begin{split}
\|\partial_t^k A^{\frac{s}{2}}e^{-tA} f \|_{L^\infty} 
& = \|e^{-\frac{t}{2}A} A^{k+\frac{s}{2}}e^{-\frac{t}{2}A} f \|_{L^\infty} \\
& \lesssim t^{-\alpha} \|A^{k+\frac{s}{2}}e^{-\frac{t}{2}A} f \|_{L^2}\\
& \lesssim t^{-\alpha-k-\frac{s}{2}} \|f \|_{L^2}
\end{split}
\end{equation}
for any $t>0$. 
Moreover, 
since $A$ is self-adjoint on $L^2(X)$, we use the duality argument to obtain
\begin{equation}\label{lem:3.1_2}
\|\partial_t^k A^{\frac{s}{2}}e^{-tA}\|_{L^1 \to L^2} \lesssim  t^{-\alpha-k-\frac{s}{2}}
\end{equation}
for any $t>0$.
Finally, by combining \eqref{lem:3.1_1} and \eqref{lem:3.1_2}, we also have 
\[
\begin{split}
\|\partial_t^k A^{\frac{s}{2}}e^{-tA} f \|_{L^\infty} 
& =
\|\partial_t^k A^{\frac{s}{2}}e^{-\frac{t}{2}A} e^{-\frac{t}{2}A} f \|_{L^\infty} \\
& \lesssim t^{-\alpha-k-\frac{s}{2}} \|e^{-\frac{t}{2}A} f \|_{L^2}\\
& \lesssim t^{-\alpha-k-\frac{s}{2}}\cdot t^{-\alpha} \| f \|_{L^1}\\
& = t^{-2\alpha-k-\frac{s}{2}} \| f \|_{L^1}
\end{split}
\]
for any $t>0$. Thus, \eqref{LpLq_1} is proved.

Next, we prove \eqref{LpLq_2}. 
Similarly, it is enough to show the four cases \eqref{eq:q1q2} for this.
We begin with a proof of the following estimate:
\begin{equation}\label{LpLq_2-1}
\|\partial_t^k A^{\frac{s}{2}}\chi (A) e^{-tA}\|_{L^{2}\to L^{\infty}} \lesssim 1,\quad 0<t<1.
\end{equation}
We use the formula \eqref{appA:formula1}:
\[
(I+A)^{-\frac{s_0}{2}} g
=
\frac{1}{\Gamma(\frac{s_0}{2})} \int_{0}^\infty
t^{\frac{s_0}{2}-1} e^{-t} e^{-tA}g\, dt
\]
with $s_0>0$. 
Then, using \eqref{ass:A}, we have
\[
\|(I+A)^{-\frac{s_0}{2}} g\|_{L^\infty}
 \lesssim  
\int_{0}^\infty
t^{\frac{s_0}{2}-1-\alpha} e^{-t} \|g\|_{L^2}\, dt
 \lesssim \|g\|_{L^2}
\]
for any $g\in L^2(X)$ and $s_0>2\alpha$. Taking $g= (I+A)^{\frac{s_0}{2}}f$, we have the Sobolev inequality
\begin{equation}\label{Sobolev-special}
\|f\|_{L^\infty} \lesssim \|f\|_{H^{s_0}(A)},\quad f\in H^{s_0}(A)
\end{equation}
for $s_0>2\alpha$. By this inequality and uniform $L^2$-boundedness of $(I + A)^{k+\frac{s+s_0}{2}}\chi (A) e^{-tA}$ with respect to $t$, we obtain
\[
\|\partial_t^k A^{\frac{s}{2}}\chi (A) e^{-tA} f \|_{L^\infty} \lesssim  \|(I + A)^{k+\frac{s+s_0}{2}} \chi (A) e^{-tA} f \|_{L^2} \lesssim  \|f\|_{L^2}
\]
for any $f\in L^2(X)$. 
This proves \eqref{LpLq_2-1}. 

By combining \eqref{LpLq_2-1} and 
\eqref{LpLq_1} with $(q_1,q_2)=(2,\infty)$, we obtain \eqref{LpLq_2} with $(q_1,q_2)=(2,\infty)$. Similarly, we can also obtain the case $(q_1,q_2)=(2,2)$. 
The remaining cases can be obtained in a similar way to the proof of \eqref{LpLq_1}.
The proof of Lemma \ref{lem:LpLq} is finished.
\end{proof}

\begin{lem}\label{lem:symbol_1}
For any $\lambda\ge0$ and $t>0$, we have
\begin{equation}\label{symbol_1}
|\partial_t^k\mathcal D(t,\lambda)| 
\lesssim 1,\quad k=0,1.
\end{equation}
\end{lem}
\begin{proof}
We write 
\[
\partial_t \mathcal D(t,\lambda)
=
-\frac{\mathcal D(t,\lambda)}{2} 
+
e^{-\frac{t}{2}}\cosh \left(t\sqrt{\frac14 - \lambda}\right)
\]
for any $0\le \lambda< 1/4$ and $t>0$. 
First, we will prove the estimate \eqref{symbol_1} for any $0\le \lambda<1/4$. 
We note that $-1/2 + \sqrt{1/4 -\lambda} \le 0$, since $0\le \lambda<1/4$. 
If $0 \le \lambda \le 1/8$, then
we have
\[
\mathcal D(t,\lambda)
=
\frac{e^{-\frac{t}{2}}( e^{t\sqrt{\frac14 - \lambda}} - e^{-t\sqrt{\frac14 - \lambda}} )}
{2\sqrt{\frac14 - \lambda}}
\le 
\frac{e^{-\frac{t}{2} + t\sqrt{\frac14 - \lambda}} }
{2\sqrt{\frac14 - \lambda}}
\le \sqrt{2} 
\]
for any $t>0$. If $1/8 \le \lambda < 1/4$, then 
we have
\[
|\mathcal D(t,\lambda)|
=
\frac{e^{-\frac{t}{2}}}{2} \cdot 
\frac{ e^{t\sqrt{\frac14 - \lambda}} - e^{-t\sqrt{\frac14 - \lambda}} }
{\sqrt{\frac14 - \lambda}} 
\le \frac{e^{-\frac{t}{2}}}{2} \cdot 
2te^{t\sqrt{\frac14 - \lambda}} 
\le t e^{-\frac{2-\sqrt{2}}{4}t}
\lesssim 1
\]
for any $t>0$, where
we used in the second step 
\[
e^{t\sqrt{\frac14 - \lambda}} - e^{-t\sqrt{\frac14 - \lambda}}
\le 
2
\left(t\sqrt{\frac14 - \lambda}\right)
e^{t\sqrt{\frac14 - \lambda}}.
\]
Moreover, 
the above estimates give
\[
|\partial_t \mathcal D(t,\lambda)|
\le 
\frac{|\mathcal D(t,\lambda)|}{2} + e^{-\frac{t}{2} + \sqrt{\frac14 -\lambda}} 
\le 
\frac{|\mathcal D(t,\lambda)|}{2} + 1
\lesssim 1
\]
for any $t>0$.
Therefore, we obtain the estimate \eqref{symbol_1} for any $0\le \lambda<1/4$. 
If $\lambda> 1/4$, then
the estimates
\[
|\mathcal D(t,\lambda)|
=
te^{-\frac{t}{2}} \cdot \frac{\left|\sin \left(t\sqrt{\lambda-\frac14} \right)\right|}{t\sqrt{\lambda-\frac14}} \le te^{-\frac{t}{2}} \lesssim 1,
\]
\[
|\partial_t \mathcal D(t,\lambda)|
\le
\frac{|\mathcal D(t,\lambda)|}{2} + e^{-\frac{t}{2}} \left|\cos \left(t\sqrt{\lambda-\frac14}\right)\right|
\le \left(\frac{t}{2} + 1\right)e^{-\frac{t}{2}}\lesssim 1
\]
hold
for any $t>0$. Therefore, we also obtain \eqref{symbol_1} for $\lambda >1/4$. 
The estimate \eqref{symbol_1} is trivial for $\lambda = 1/4$. The proof of Lemma \ref{lem:symbol_1} is finished.
\end{proof}

\begin{lem}\label{lem:symbol_2}
For any $t\ge1$ and $k=0,1$, we have
\begin{equation}\label{symbol_2_1}
\|
\mathbf{1}_{[0,\frac18]}(A) e^{\frac{t}{2}A} \partial_t^k (\mathcal D(t,A) - e^{-tA})
\|_{L^2 \to L^2}
\lesssim \langle t \rangle^{-1-k},
\end{equation}
\begin{equation}\label{symbol_2_2}
\|
\mathbf{1}_{(\frac18,\frac14]}(A) e^{\frac{t}{2}A} \partial_t^k (\mathcal D(t,A) - e^{-tA})
\|_{L^2 \to L^2}
\lesssim \langle t \rangle^{-1-k}.
\end{equation}
Moreover, for $s\ge 0$, we have 
\begin{equation}\label{symbol_2_3}
\|
A^{\frac{s}{2}}\mathbf{1}_{(\frac14,1]}(A) e^{\frac{t}{2}A}\partial_t^k (\mathcal D(t,A) - e^{-tA})
\|_{L^2 \to H^{\beta}(A)}
\lesssim e^{-\frac{t}{4}},
\end{equation}
\begin{equation}\label{symbol_2_4}
\|
A^{\frac{s}{2}}\mathbf{1}_{(1,\infty)}(A) e^{\frac{t}{2}A}\partial_t^k (\mathcal D(t,A) - e^{-tA})
\|_{L^2 \to H^{\beta}(A)}
\lesssim e^{-\frac{t}{2}}
\end{equation}
with $\beta =s + k-1$.
\end{lem}

\begin{proof}
The estimates \eqref{symbol_2_3} and \eqref{symbol_2_4} immediately follows from the inequalities 
\[
\begin{split}
|\lambda^{\frac{s}{2}} \partial_t^k (\mathcal D(t,\lambda) - e^{-t\lambda})|
 \lesssim e^{-\frac{t}{2}} (1+\lambda)^{\frac{\beta}{2}}
 \quad \text{for $\lambda>1$}
\end{split}
\]
with $\beta =s + k-1$. 
The estimate \eqref{symbol_2_2} is similarly proved to \eqref{symbol_2_1}, and therefore, we will give only a proof of \eqref{symbol_2_1}.

For the case $k=0$, based on the proof of the first inequality of Lemma 2.9 in \cite{IIOW-2019}, we have 
\begin{equation}\label{cal_1}
\left|
\frac{1}{\sqrt{1-4\lambda}} -1 
\right|
\lesssim \lambda,
\end{equation}
\begin{equation}\label{cal_2}
\left|
e^{\frac{t}{2}\lambda}(e^{t(-\frac12 +\sqrt{\frac14 - \lambda})} - e^{-t\lambda} )
\right|
\lesssim 
t \lambda^2 e^{-\frac{t}{2}\lambda} \mathbf{1}_{[0,\langle t\rangle^{-\frac12}]}(\lambda)
+ e^{-\frac{t}{2}\langle t\rangle^{-\frac12}}
\mathbf{1}_{(\langle t\rangle^{-\frac12}, \infty)}(\lambda)
\end{equation}
for $0 \le \lambda \le 1/8$.
In fact, we calculate \eqref{cal_1} as
\[
\left|
\frac{1}{\sqrt{1-4\lambda}} -1 
\right|
=
\left|
\frac{1}{\sqrt{1-4\lambda}} - \frac{1+ \sqrt{1-4\lambda}}{1+ \sqrt{1-4\lambda}} 
\right|
= 
\frac{4\lambda}{\sqrt{1-4\lambda} (1+ \sqrt{1-4\lambda})} 
\lesssim \lambda
\]
for $0 \le \lambda \le 1/8$. From this, we also have 
\begin{equation}\label{cal_1'}
\frac14 \left|
1 - \sqrt{1-4\lambda}
\right|^2
=
\left(
\frac{2\lambda}{1+\sqrt{1-4\lambda}}
\right)^2.
\end{equation}
As for \eqref{cal_2}, we write 
\[
\begin{split}
\left|
e^{\frac{t}{2}\lambda}(e^{t(-\frac12 +\sqrt{\frac14 - \lambda})} - e^{-t\lambda} )
\right|
& =
e^{-\frac{t}{2}\lambda}
\left|
e^{t [ \lambda + \sqrt{\frac14 - \lambda} -\frac12]} -1
\right|\\
& =
e^{-\frac{t}{2}\lambda}
\left|
e^{-\frac{t}{4} \left|
1 - \sqrt{1-4\lambda}
\right|^2} -1
\right|\\
& =
e^{-\frac{t}{2}\lambda}
\left|
e^{
-\frac{4t\lambda^2}{(1 + \sqrt{1-4\lambda)^2}}} -1
\right|,
\end{split}
\]
where we used \eqref{cal_1'} in the last step.
When $0\le \lambda \le \langle t\rangle^{-\frac12}$ and $0\le \lambda \le 1/8$,
by the mean value theorem, we estimate 
\[
\left|
e^{
-\frac{4t\lambda^2}{(1 + \sqrt{1-4\lambda)^2}}} -1
\right|
\lesssim 
\frac{4t\lambda^2}{(1 + \sqrt{1-4\lambda)^2}} e^{-\lambda_0}
\lesssim t\lambda^2 
\]
where $\lambda_0$ is a constant depending on $t$ and $\lambda$ such that 
\[
0< \lambda_0 < \frac{4t\lambda^2}{(1 + \sqrt{1-4\lambda)^2}} \lesssim 1.
\]
Hence, we conclude
\[
e^{-\frac{t}{2}\lambda}
\left|
e^{
-\frac{4t\lambda^2}{(1 + \sqrt{1-4\lambda)^2}}} -1
\right|
\lesssim t\lambda^2 e^{-\frac{t}{2}\lambda}
\]
for $0 \le \lambda \le 1/8$ and $\lambda \le \langle t\rangle^{-\frac12}$.
On the other hand,
when $\lambda > \langle t\rangle^{-\frac12}$
and $0 \le \lambda \le 1/8$, we have
\[
e^{-\frac{t}{2}\lambda}
\left|
e^{
-\frac{4t\lambda^2}{(1 + \sqrt{1-4\lambda)^2}}} -1
\right|
\lesssim e^{-\frac{t}{2}\langle t\rangle^{-\frac12}}.
\]
Combining the above three, we obtain \eqref{cal_2}. 
By the triangle inequality, \eqref{cal_1} and \eqref{cal_2}, we have
\[
\begin{split}
\left|
e^{\frac{t}{2}\lambda}(\mathcal D(t,\lambda) - e^{-t\lambda} )
\right|
& \lesssim 
e^{-\frac{t}{2}\lambda}
\left|
\frac{1}{\sqrt{1-4\lambda}} -1 
\right|
+
\left|
e^{\frac{t}{2}\lambda}(e^{t(-\frac12 +\sqrt{\frac14 - \lambda})} - e^{-t\lambda} )
\right|\\
&\quad
+ \frac{e^{\frac{t}{2}\lambda - \frac{t}{2}-t\sqrt{1-4\lambda}}}{\sqrt{1-4\lambda}} \\
& \lesssim 
(\lambda + \langle t \rangle \lambda^2)e^{-\frac{t}{2}\lambda} + e^{-\frac{t}{2}\langle t\rangle^{-\frac12}}\mathbf{1}_{(\langle t\rangle^{-\frac12}, \infty)}(\lambda)
+ e^{-\frac{t}{4}}
\end{split}
\]
for $0 \le \lambda \le 1/8$.
Since $d\|E_A(\lambda)f\|_{L^2}^2$ is a positive measure,
we derive from this inequality that 
\[
\begin{split}
& \|
\mathbf{1}_{[0,\frac18)}(A) e^{\frac{t}{2}A}(\mathcal D(t,A) - e^{-tA})f
\|_{L^2}^2\\
& \le  
\int_0^{\frac18} |e^{\frac{t}{2}\lambda}(\mathcal D(t,\lambda) - e^{-t\lambda})|^2\, d\|E_A(\lambda) f\|_{L^2}^2\\
& \lesssim 
\int_0^{\frac18} \left|
(\lambda + \langle t \rangle \lambda^2)e^{-\frac{t}{2}\lambda} + e^{-\frac{t}{2}\langle t\rangle^{-\frac12}}\mathbf{1}_{(\langle t\rangle^{-\frac12}, \infty)}(\lambda)
+ e^{-\frac{t}{4}}
\right|^2\, d\|E_A(\lambda) f\|_{L^2}^2\\
& \lesssim 
\langle t \rangle^{-2}\|f\|_{L^2}^2,
\end{split}
\]
where we used the following inequalities in the last step:
\[
(\lambda + \langle t \rangle \lambda^2)e^{-\frac{t}{2}\lambda}
\lesssim \langle t \rangle^{-1}
\quad \text{and}\quad 
e^{-\frac{t}{2}\langle t\rangle^{-\frac12}}
\lesssim   \langle t \rangle^{-1}
\]
for any $0\le \lambda \le 1/8$ and $t\ge 0$. 

For the case $k=1$, 
we use \eqref{cal_1'} and write 
\[
\begin{split}
e^{\frac{t}{2}\lambda} \partial_t (\mathcal D(t,\lambda) - e^{-t\lambda})
& = 
-\lambda e^{-\frac{t}{2}\lambda} \left[
a(\lambda) e^{t [ \lambda + \sqrt{\frac14 - \lambda} -\frac12]} -1
\right]
+
b(\lambda) e^{-\frac{1-\lambda}{2} t - t \sqrt{\frac14 - \lambda}}\\
& = 
-\lambda e^{-\frac{t}{2}\lambda} \left[
a(\lambda) e^{
-\frac{4t\lambda^2}{(1 + \sqrt{1-4\lambda)^2}}} -1
\right]
+
b(\lambda) e^{-\frac{1-\lambda}{2} t - t \sqrt{\frac14 - \lambda}}
\end{split}
\]
with 
\[
a(\lambda) := \frac{1}{2(\frac14-\lambda) + \sqrt{\frac14 -\lambda}},
\quad 
b(\lambda):= \frac{1}{4\sqrt{\frac14 -\lambda}} + \frac12,
\]
We give only a proof of the inequality 
\begin{equation}\label{k=1:goal}
\lambda e^{-\frac{t}{2}\lambda} \left|
a(\lambda) e^{
-\frac{4t\lambda^2}{(1 + \sqrt{1-4\lambda)^2}}} -1
\right| \lesssim \langle t\rangle ^{-2}
\end{equation}
for any $0 \le \lambda \le 1/8$, $\lambda \le \langle t\rangle ^{-1}$ and $t\ge 1$. 
Indeed, the other symbol estimates have an exponential decay in $t\ge1$, 
and hence, once \eqref{k=1:goal} is shown, the same argument as in the case $k=0$ gives the required estimate \eqref{symbol_2_1} with $k=1$. 
Therefore, it is enough to show \eqref{k=1:goal}.

Set 
\[
g_t(\lambda)
:= -\frac{4t\lambda^2}{(1 + \sqrt{1-4\lambda)^2}}
\quad \text{and}\quad 
F_t(\lambda) := a(\lambda) e^{g_t(\lambda)}.
\]
Then $F_t(0) = a(0) = 1$ and it follows from the mean value theorem that
\[
\left|
a(\lambda) e^{
-\frac{4t\lambda^2}{(1 + \sqrt{1-4\lambda)^2}}} -1
\right|
=
\left|
F_t(\lambda) - F_t (0)
\right|
\le \lambda | F'_t(\lambda_0)|
\]
for some $\lambda_0 \in (0,\lambda)$ depending on $t$ and $\lambda$.
We calculate 
\[
\begin{split}
F'_t(\lambda) 
& =
a'(\lambda) e^{g_t(\lambda)} + a(\lambda) \left(\frac{d}{d\lambda} g_t(\lambda)\right) e^{g_t(\lambda)}\\
& =
a'(\lambda) e^{g_t(\lambda)} + a(\lambda) \left(
- \frac{4t\lambda}{(1+\sqrt{\frac14 - \lambda})\sqrt{\frac14-\lambda}}
\right) e^{g_t(\lambda)}.
\end{split}
\]
Noting $a(\lambda), a'(\lambda)$ are bounded in $\lambda \in [0,1/8]$ and $e^{g_t(\lambda)}$ is bounded in $t\ge1$ and $\lambda \in [0,1/8]$, we have 
\[
|F'_t(\lambda_0) |\lesssim 1 + t\lambda
\]
for $0\le \lambda \le 1/8$, $\lambda \le \langle t\rangle ^{-1}$ and $t\ge1$.
Therefore we obtain
\[
\lambda e^{-\frac{t}{2}\lambda} \left|
a(\lambda) e^{
-\frac{4t\lambda^2}{(1 + \sqrt{1-4\lambda)^2}}} -1
\right| \lesssim
\lambda e^{-\frac{t}{2}\lambda}\cdot \lambda (1 + t\lambda)
\lesssim
 \langle t\rangle ^{-2}
\]
for $0\le \lambda \le 1/8$, $\lambda \le \langle t\rangle ^{-1}$ and $t\ge1$, which proves \eqref{k=1:goal}.
The proof of Lemma \ref{lem:symbol_2} is finished.
\end{proof}

\begin{proof}[Proof of Theorem \ref{thm:Matsumura-est}]
We give only proofs of \eqref{est:q-infty-asy} and \eqref{est:q-2-asy}. 
The proofs of estimates \eqref{Mat-est1} and \eqref{Mat-est2} are similar, and so we may omit them.

Let us decompose the spectrum $\sigma(A)$ into
\[
\sigma(A) = \bigcup_{i=1}^{4} \sigma_i(A),
\]
where $\sigma_1(A) = \sigma(A) \cap [0, 1/8]$, $\sigma_2(A) = \sigma(A) \cap (1/8, 1/4]$, $\sigma_3(A) = \sigma(A) \cap (1/4,1]$ and $\sigma_4(A) = \sigma(A) \cap (1,\infty)$.
Let $f\in L^q(X) \cap H^\beta(A)$. 
Since $\partial_t^k (\mathcal D(t,\lambda) - e^{-t\lambda}) \in L^\infty(\sigma(A))$ for each $t>0$ by Lemma \ref{lem:symbol_1}, we have 
\[
\partial_t^kA^{\frac{s}{2}} (\mathcal D(t,A) - e^{-tA})f = \partial_t^k (\mathcal D(t,A) - e^{-tA}) A^{\frac{s}{2}}f \in L^2(X),
\]
and write
\[
\langle \partial_t^k A^{\frac{s}{2}} (\mathcal D(t,A) - e^{-tA}) f, g \rangle_{L^2}
=
I_1 + I_2 + I_3 + I_4
\]
for any $g \in L^1(X) \cap L^2(X)$, 
where
\[
I_i :=
\partial_t^k
\int_{\sigma_i(A)} \lambda^{\frac{s}{2}} (\mathcal D(t,\lambda) - e^{-t\lambda})\, d\langle E_A(\lambda) f, g\rangle_{L^2},\quad i=1,2,3,4.
\]

First, we give a proof of \eqref{est:q-infty-asy}. 
By the dominated convergence theorem, we write the term $I_1$ as follows:
\[
\begin{split}
I_1 
&= 
\partial_t^k
\int_{\sigma_1(A)} \lambda^{\frac{s}{2}} (\mathcal D(t,\lambda) - e^{-t\lambda})\, d\langle E_A(\lambda) f, g\rangle_{L^2}\\
&= 
\int_{\sigma_1(A)} \lambda^{\frac{s}{2}} \partial_t^k(\mathcal D(t,\lambda) - e^{-t\lambda})\, d\langle E_A(\lambda) f, g\rangle_{L^2}\\
& = 
\int_{\sigma(A)} \lambda^{\frac{s}{2}} e^{-\frac{t}{4} \lambda} \mathbf{1}_{[0,\frac18]}(\lambda) \cdot \mathbf{1}_{[0,\frac18]}(\lambda) e^{\frac{t}{2}\lambda} \partial_t^k (\mathcal D(t,\lambda) - e^{-t\lambda}) \mathbf{1}_{[0,\frac18]}(\lambda) e^{-\frac{t}{4}\lambda}\, d\langle E_A(\lambda) f, g\rangle_{L^2}\\
& =
\langle A^{\frac{s}{2}} e^{-\frac{t}{4} A} \mathbf{1}_{[0,\frac18]}(A) \mathcal D_1(t,A) \mathbf{1}_{[0,\frac18]}(A) e^{-\frac{t}{4} A} f, g \rangle_{L^2},
\end{split}
\]
where 
\[
\mathcal D_1(t,\lambda) :=  \mathbf{1}_{[0,\frac18]}(\lambda) e^{\frac{t}{2}\lambda} \partial_t^k(\mathcal D(t,\lambda) - e^{-t\lambda}).
\]
Then, by Lemmas \ref{lem:LpLq} and \ref{lem:symbol_2}, we estimate
\[
\begin{split}
|I_1| 
& \le \|A^{\frac{s}{2}} e^{-\frac{t}{4} A} \mathbf{1}_{[0,\frac18]}(A) \mathcal D_1(t,A) \mathbf{1}_{[0,\frac18]}(A) e^{-\frac{t}{4} A} f\|_{L^\infty(X)}\|g\|_{L^1(X)}\\
& \le \| A^{\frac{s}{2}} e^{-\frac{t}{4} A} \mathbf{1}_{[0,\frac18]}(A) \|_{L^2\to L^\infty}
\|\mathcal D_1(t,A)\|_{L^2 \to L^2}
\|\mathbf{1}_{[0,\frac18]}(A) e^{-\frac{t}{4} A}\|_{L^q \to L^2}
\|f\|_{L^q}\|g\|_{L^1}\\
& \lesssim  \langle t \rangle^{-\alpha-\frac{s}{2}} \cdot \langle t \rangle^{-1-k}\cdot \langle t \rangle^{-2\alpha (\frac{1}{q}-\frac12)} \|f\|_{L^q} \|g\|_{L^1}\\
& \lesssim  \langle t \rangle^{-\frac{2\alpha}{q}-\frac{s}{2}-k-1} \|f\|_{L^q} \|g\|_{L^1}
\end{split}
\]
for any $f \in L^q(X)$ and $g \in L^1(X)\cap L^2(X)$.
Similarly, we have
\[
|I_2| \lesssim  \langle t \rangle^{-\frac{2\alpha}{q}-\frac{s}{2}-k-1}\|f\|_{L^q} \|g\|_{L^1}.
\]
Next, we estimate the term $I_4$. By the dominated convergence theorem, we write
\[
\begin{split}
I_4 
& = 
\partial_t^k
\int_{\sigma_4(A)} \lambda^{\frac{s}{2}} (\mathcal D(t,\lambda) - e^{-t\lambda}) \, d\langle E_A(\lambda) f, g\rangle_{L^2}\\
& = 
\int_{\sigma(A)} \lambda^{\frac{s}{2}}\mathbf{1}_{(1,\infty)}(\lambda) \partial_t^k (\mathcal D(t,\lambda) - e^{-t\lambda})\, d\langle E_A(\lambda) f, g\rangle_{L^2}\\
& = 
\langle \mathcal D_4(t,A) 
 f, g \rangle_{L^2},
\end{split}
\]
where 
\[
\mathcal D_4(t,\lambda):= \lambda^{\frac{s}{2}} \mathbf{1}_{(1,\infty)}(\lambda) \partial_t^k (\mathcal D(t,\lambda) - e^{-t\lambda}).
\]
By Lemma \ref{lem:Sobolev}, we have
\[
|I_4|
\le \|\mathcal D_4(t,A) f\|_{L^\infty}\|g\|_{L^1} \lesssim  \|\mathcal D_4(t,A) f\|_{H^{s_0}(A)}\|g\|_{L^1}
\]
for $s_0>2\alpha$. 
By Lemma \ref{lem:symbol_2},
we estimate
\[
\begin{split}
\|\mathcal D_4(t,A) f\|_{H^{s_0}(A)}
\lesssim e^{-\frac{t}{2}}\|f\|_{H^{\beta}(A)}
\end{split}
\]
with $\beta = s_0 +s + k-1$.
Hence, 
\[
|I_4| \lesssim  e^{-\frac{t}{2}}\|f\|_{H^{\beta}(A)}\|g\|_{L^1}.
\]
Similarly, we have
\[
|I_3| \lesssim  e^{-\frac{t}{4}}\|f\|_{H^{\beta}(A)}\|g\|_{L^1}.
\]
Summarizing the estimates for $I_1$, $I_2$, $I_3$ and $I_4$, we conclude \eqref{est:q-infty-asy}. 

Next, we give a proof of \eqref{est:q-2-asy}.
By Lemmas \ref{lem:LpLq} and \ref{lem:symbol_2}, we estimate $I_1$ and $I_4$ as
\[
\begin{split}
|I_1| 
& \le \|A^{\frac{s}{2}} e^{-\frac{t}{4} A} \mathbf{1}_{[0,\frac18]}(A) \mathcal D_1(t,A) \mathbf{1}_{[0,\frac18]}(A) e^{-\frac{t}{4} A} f\|_{L^2(X)}\|g\|_{L^2(X)}\\
& \le \| A^{\frac{s}{2}} e^{-\frac{t}{4} A} \mathbf{1}_{[0,\frac18]}(A) \|_{L^2\to L^2}
\|\mathcal D_1(t,A)\|_{L^2 \to L^2}
\|\mathbf{1}_{[0,\frac18]}(A) e^{-\frac{t}{4} A}\|_{L^q \to L^2}
\|f\|_{L^q}\|g\|_{L^2}\\
& \lesssim  \langle t \rangle^{-\frac{s}{2}} \cdot \langle t \rangle^{-1-k}\cdot \langle t \rangle^{-2\alpha (\frac{1}{q}-\frac12)} \|f\|_{L^q} \|g\|_{L^1}\\
& \lesssim  \langle t\rangle^{-2\alpha(\frac{1}{q} -\frac12) - k - \frac{s}{2}-1}  \|f\|_{L^q} \|g\|_{L^2}
\end{split}
\]
for any $f \in L^q(X)$ and $g \in L^2(X)$, 
and
\[
|I_4| 
 \le \|\mathcal D_4(t,A) f\|_{L^2}\|g\|_{L^2}
 \lesssim e^{-\frac{t}{2}} \|f\|_{H^{s + k-1}(A)}\|g\|_{L^2}
\]
for any $f \in H^{s + k-1}(A)$ and $g \in L^2(X)$.
The estimates for $I_2$ and $I_3$ are similar. 
The proof of Theorem \ref{thm:Matsumura-est} is complete.
\end{proof}


\section{Proof of Theorem \ref{thm:sdge}}
Throughout this section, we suppose that $A$ satisfies Assumptions \ref{assum:A} and \ref{assum:B}. 
Based on the result on 
the local existence of mild solutions to the problem \eqref{ndw} (Proposition \ref{prop:local}), 
we will prove Theorem \ref{thm:sdge}. 
The proof of Proposition \ref{prop:local} is based on the standard fixed point argument and uses the contraction property \eqref{condi-F} of $F$, and so we may omit the proof. \\

To prove Theorem \ref{thm:sdge}, 
we aim to prove a priori estimate
\[
\|u\|_{X(T)} \le C
\]
for some constant $C>0$ independent of $T$, where
\[
\begin{split}
\|\phi\|_{X(T)} :=
\sup_{t\in (0,T)} 
\Big\{
& \langle t\rangle^{2\alpha ( \frac{1}{q} - \frac12) + 1}
(\log(2+t))^{-\delta} \|\partial_t \phi(t)\|_{L^2} \\
& 
+ \langle t\rangle^{2\alpha ( \frac{1}{q} - \frac12) + \frac12}
\|A^\frac12 \phi(t)\|_{L^2}
+
\langle t\rangle^{2\alpha ( \frac{1}{q} - \frac12)}
\|\phi(t)\|_{L^2}
\Big\}
\end{split}
\]
with
\begin{equation}\label{delta}
\delta = 
\begin{cases}
0\quad &\text{if } q < 2\alpha (p-1),\\
1 & \text{if } q = 2\alpha (p-1).
\end{cases}
\end{equation}

Then, we have the following estimate on
$X(T)$.
\begin{lem}\label{lem:interpolation}
For any $r\in [2,\infty]$ satisfying $\alpha (1/2 - 1/r) \le 1/4$
except for $(r,\alpha) = (\infty,1/2)$,
we have
\[
\sup_{t \in (0,T)} 
\langle t \rangle^{2\alpha (\frac{1}{q} - \frac{1}{r})}\|\phi(t)\|_{L^r} 
\lesssim \|\phi\|_{X(T)}.
\]
Here, the implicit constant does not
depend on $T$.
\end{lem}
\begin{proof}
The assertion immediately follows from Lemmas \ref{lem:GN} and \ref{lem:cS} and the definition of $X(T)$. 
\end{proof}

To prepare the estimate for the Duhamel term, we define
\[
\|\psi\|_{Y(T)} :=
\sup_{t\in (0,T)} 
\left\{
\langle t\rangle^{2\alpha ( \frac{1}{q} - \frac{1}{2p})p}
\|\psi(t)\|_{L^2} + 
\langle t\rangle^{2\alpha ( \frac{1}{q} - \frac{1}{\kappa p})p}
\|\psi(t)\|_{L^\kappa}
\right\}
\]
with
$\kappa = \max \{1, 2/p\}$.

\begin{lem}\label{lem:keyest1}
Let $q \in [1,2]$ and $p\in (1,\infty)$ satisfy \eqref{supercritical}. 
Then 
\[
\left\| 
\int_0^t \mathcal{D}(t - \tau, A) \psi(\tau)\, d\tau
\right\|_{X(T)}
\lesssim \|\psi\|_{Y(T)}.
\]
\end{lem}

\begin{proof}
First, we estimate the third quantity of the Duhamel term in $X(T)$, i.e.,
\[
\langle t\rangle^{2\alpha ( \frac{1}{q} - \frac12)}
\left\|\int_0^t \mathcal{D}(t - \tau, A) \psi(\tau)\, d\tau\right\|_{L^2}.
\]
This is bounded from above by 
\[
\langle t\rangle^{2\alpha ( \frac{1}{q} - \frac12)}
\left(\int_0^\frac{t}{2} + \int_\frac{t}{2}^t\right)
\left\| \mathcal{D}(t - \tau, A) \psi(\tau)\right\|_{L^2}\, d\tau\\
 =: I + II.
\]
By \eqref{Mat-est2} from Theorem \ref{thm:Matsumura-est} and the embedding $L^2(X) \hookrightarrow H^{-1}(A)$, we estimate
\[
\begin{split}
I 
& \lesssim 
\langle t\rangle^{2\alpha ( \frac{1}{q} - \frac12)}
\int_0^\frac{t}{2} 
\left(
\langle t-\tau\rangle^{-2\alpha ( \frac{1}{\kappa} - \frac12)}
\|\psi(\tau)\|_{L^\kappa}
+ e^{-\frac{t-\tau}{4}}\|\psi(\tau)\|_{H^{-1}(A)}
\right)\, d\tau\\
& \lesssim 
\langle t\rangle^{-2\alpha ( \frac{1}{\kappa} - \frac{1}{q})}
\int_0^\frac{t}{2} 
\langle \tau\rangle^{-2\alpha ( \frac{1}{q} - \frac{1}{\kappa p})p}
\, d\tau \cdot 
\|\psi\|_{Y(T)}\\
& \qquad
+ e^{-\frac{t}{8}}\langle t\rangle^{2\alpha ( \frac{1}{q} - \frac12)}
\int_0^\frac{t}{2} 
\langle \tau\rangle^{-2\alpha ( \frac{1}{q} - \frac{1}{2p})p}
\, d\tau \cdot 
\|\psi\|_{Y(T)}.
\end{split}
\]
Here, the assumption
\eqref{supercritical}
implies that,
when $\kappa < q$,
\begin{align*}
    \langle t\rangle^{-2\alpha ( \frac{1}{\kappa} - \frac{1}{q})}
    \int_0^\frac{t}{2} 
    \langle \tau\rangle^{-2\alpha ( \frac{1}{q} - \frac{1}{\kappa p})p} \, d\tau
    &\lesssim
    \begin{cases}
        \langle t \rangle^{-2\alpha\left( \frac{1}{\kappa}-\frac{1}{q} \right)}
        &\left( -2\alpha \left( \frac{1}{q}-\frac{1}{\kappa p} \right)p < -1 \right),\\
        \langle t \rangle^{-2\alpha\left( \frac{1}{\kappa}-\frac{1}{q} \right)}
        \log(2+t)
        &\left( -2\alpha \left( \frac{1}{q}-\frac{1}{\kappa p} \right)p = -1 \right),\\
        \langle t \rangle^{-\frac{2\alpha}{q} (p-1) + 1}
        &\left( -2\alpha \left( \frac{1}{q}-\frac{1}{\kappa p} \right)p > -1 \right)
    \end{cases}\\
    &\lesssim 1
\end{align*}
holds; when
$\kappa = q$,
\begin{align*}
     \langle t\rangle^{-2\alpha ( \frac{1}{\kappa} - \frac{1}{q})}
    \int_0^\frac{t}{2} 
    \langle \tau\rangle^{-2\alpha ( \frac{1}{q} - \frac{1}{\kappa p})p} \, d\tau
    &=
    \int_0^{\frac{t}{2}} \langle \tau \rangle^{-\frac{2\alpha}{q}(p-1)} \,d\tau
    \lesssim 1.
\end{align*}
Here, we note that
$\kappa = q = \max\{ 1, \frac{2}{p} \}$
and the conditions
$1+\frac{1}{2\alpha} < p$
and
$\frac{2}{p} < 2\alpha (p-1)$
lead to
$q < 2\alpha (p-1)$.
Therefore, we conclude
\begin{align*}
    I \lesssim \|\psi\|_{Y(T)}.
\end{align*}
Next, we have
\[
\begin{split}
II 
& \lesssim 
\langle t\rangle^{2\alpha ( \frac{1}{q} - \frac12)}
\int_\frac{t}{2}^t
\|\psi(\tau)\|_{L^2}\, d\tau\\
& \lesssim 
\langle t\rangle^{2\alpha ( \frac{1}{q} - \frac12)}
\int_\frac{t}{2}^t
\langle \tau\rangle^{-2\alpha ( \frac{1}{q} - \frac{1}{2p})p}
\, d\tau \cdot 
\|\psi\|_{Y(T)}\\
& \lesssim \|\psi\|_{Y(T)},
\end{split}
\]
where we used
\eqref{supercritical}.
Thus, we obtain 
\[
\langle t\rangle^{2\alpha ( \frac{1}{q} - \frac12)}
\left\|\int_0^t \mathcal{D}(t - \tau, A) \psi(\tau)\, d\tau\right\|_{L^2}
\lesssim \|\psi\|_{Y(T)}.
\]
Similarly, the second quantity in $X(T)$ is estimated by  
\[
\langle t\rangle^{2\alpha ( \frac{1}{q} - \frac12) + \frac12}
\left\| A^\frac12
\int_0^t \mathcal{D}(t - \tau, A) \psi(\tau)\, d\tau
\right\|_{L^2}
\lesssim
\|\psi\|_{Y(T)}.
\]

Finally, we will estimate the first quantity in $X(T)$, i.e.,
\[
\langle t\rangle^{2\alpha ( \frac{1}{q} - \frac12) + 1}
(\log(2+t))^{-\delta} \left\|\partial_t \int_0^t \mathcal{D}(t - \tau, A) \psi(\tau)\, d\tau\right\|_{L^2}.
\]
This is bounded from above by 
\[
\langle t\rangle^{2\alpha ( \frac{1}{q} - \frac12) + 1}
(\log(2+t))^{-\delta} \left( \int_0^\frac{t}{2} + \int_\frac{t}{2}^t \right)
\left\|\partial_t \mathcal{D}(t - \tau, A) \psi(\tau)\right\|_{L^2}\, d\tau =: III + IV.
\]
By Theorem \ref{thm:Matsumura-est}, we estimate
\[
\begin{split}
III 
& \lesssim 
\langle t\rangle^{2\alpha ( \frac{1}{q} - \frac12) + 1}
(\log(2+t))^{-\delta} \\
& \qquad \times 
\int_0^\frac{t}{2} 
\left(
\langle t-\tau\rangle^{-2\alpha ( \frac{1}{\kappa} - \frac12)-1}
\|\psi(\tau)\|_{L^\kappa}
+ e^{-\frac{t-\tau}{4}}\|\psi(\tau)\|_{L^2}
\right)\, d\tau\\
& \lesssim 
\langle t\rangle^{-2\alpha ( \frac{1}{\kappa} - \frac{1}{q})}(\log(2+t))^{-\delta} 
\int_0^\frac{t}{2} 
\langle \tau\rangle^{-2\alpha ( \frac{1}{q} - \frac{1}{\kappa p})p}
\, d\tau \cdot 
\|\psi\|_{Y(T)}\\
& \qquad
+ e^{-\frac{t}{8}}\langle t\rangle^{2\alpha ( \frac{1}{q} - \frac12)}(\log(2+t))^{-\delta} 
\int_0^\frac{t}{2} 
\langle \tau\rangle^{-2\alpha ( \frac{1}{q} - \frac{1}{2p})p}
\, d\tau \cdot 
\|\psi\|_{Y(T)}.
\end{split}
\]
In the same way as in the estimate of $I$,
we have
\begin{align*}
    III \lesssim \|\psi\|_{Y(T)}.
\end{align*}
Finally, applying Theorem \ref{thm:Matsumura-est}, we have
\[
\begin{split}
IV 
& \lesssim
\langle t\rangle^{2\alpha ( \frac{1}{q} - \frac12) + 1}
(\log(2+t))^{-\delta} 
\int_\frac{t}{2}^t
(\langle t-\tau\rangle^{-1} + e^{-\frac{t-\tau}{4}}) \|\psi(\tau)\|_{L^2}\, d\tau\\
& \lesssim
\langle t\rangle^{2\alpha ( \frac{1}{q} - \frac12) + 1}
(\log(2+t))^{-\delta} 
\int_\frac{t}{2}^t
\langle t-\tau\rangle^{-1} 
\langle \tau\rangle^{-2\alpha ( \frac{1}{q} - \frac{1}{2 p})p}
\, d\tau \cdot \|\psi\|_{Y(T)}\\
& \lesssim
\langle t\rangle^{-\frac{2\alpha}{q} (p-1) + 1} (\log(2+t))^{1-\delta} \cdot \|\psi\|_{Y(T)}.
\end{split}
\]
From the definition of
$\delta$ in \eqref{delta},
one has
\begin{align*}
	\langle t\rangle^{-\frac{2\alpha}{q} (p-1) + 1} (\log(2+t))^{1-\delta}
	&=
	\begin{cases}
	\langle t\rangle^{-\frac{2\alpha}{q} (p-1) + 1} (\log(2+t))
	&(q < 2\alpha (p-1)),\\
	1
	&(q = 2\alpha (p-1)),
	\end{cases}
\end{align*}
which yields
\begin{align*}
	IV \lesssim  \|\psi\|_{Y(T)}.
\end{align*}
Thus, we conclude Lemma~\ref{lem:keyest1}.
\end{proof}

By combining 
the definitions of $X(T)$ and $Y(T)$ and
Lemma \ref{lem:interpolation},
we have the following statement.

\begin{lem}\label{lem:keyest2}
Let $p \in (1,\infty)$ satisfy \eqref{additional-assum}.
Then
\[
\big\| F(u) \big\|_{Y(T)}
\lesssim \|u\|_{X(T)}^p.
\]
\end{lem}
\begin{proof}
The assumptions
$p >1$ and \eqref{additional-assum} imply
$2p \in (2, \infty)$ and 
$\alpha (\frac{1}{2} - \frac{1}{2p} ) \le \frac{1}{4}$.
Therefore, by Lemma \ref{lem:interpolation}, we have
\begin{align*}
	\langle t\rangle^{2\alpha ( \frac{1}{q} - \frac{1}{2p})p}
	\| |u(t)|^p \|_{L^2}
	=
	\left(
	\langle t\rangle^{2\alpha ( \frac{1}{q} - \frac{1}{2p} )}
	\| u(t) \|_{L^{2p}}
	\right)^p
	\lesssim
	\| u \|_{X(T)}^p.
\end{align*}
Similarly, recalling $\kappa = \max \{1, 2/p\}$
and the assumption \eqref{additional-assum}, we can easily check that
$\kappa p \in [2,\infty)$
and
$\alpha (\frac{1}{2} - \frac{1}{\kappa p} ) \le \frac{1}{4}$ hold.
Thus, Lemma \ref{lem:interpolation} leads to 
\begin{align*}
	\langle t\rangle^{2\alpha ( \frac{1}{q} - \frac{1}{\kappa p})p}
	\| |u(t)|^p \|_{L^{\kappa}}
	=
	\left(
	\langle t\rangle^{2\alpha ( \frac{1}{q} - \frac{1}{\kappa p})}
	\| u(t) \|_{L^{\kappa p}}
	\right)^p
	\lesssim
	\| u \|_{X(T)}^p.
\end{align*}
The proof is complete.
\end{proof}

We are now in a position to prove Theorem \ref{thm:sdge}. 

\begin{proof}[Proof of Theorem \ref{thm:sdge}]
Let $(u_0, u_1) \in (H^1(A) \cap L^q(X) )\times (L^2(X)\cap L^q(X))$ satisfy
\[
I_{0} = I_{0}(u_0,u_1) := \|u_0\|_{L^q} + \|u_0\|_{H^1(A)} +  \|u_1\|_{L^q} + \|u_1\|_{L^2} \le \epsilon_0.
\]
Suppose $T_{\max} < \infty$. By Proposition \ref{prop:local}, we have
\begin{equation}\label{contra}
\lim_{t \to T_{\max}}
\left(
\|u(t)\|_{H^1(A)} + \|\partial_t u(t)\|_{L^2} 
\right)=\infty.
\end{equation}
On the other hand, 
let $T \in (0, T_{\max})$, and we write
\[
u(t) = u_L(t) + \int_0^t \mathcal D(t-\tau, A)(F(u(\tau)))\, d\tau,
\]
where $u_L=u_L(t)$ is the solution to the linear problem \eqref{dw} with initial data $(u_L(0), \partial_t u_L(0)) = (u_0,u_1)$. 
Then, 
by Lemmas \ref{lem:keyest1} and \ref{lem:keyest2}, we have
\[
\|u\|_{X(T)} \le 
C I_0 + C' \|F(u)\|_{Y(T)}
\le 
C I_0 + C'' \|u\|_{X(T)}^p.
\]
Then there exists $\epsilon_0>0$ such that  
\begin{equation}\label{est}
\|u\|_{X(T)} \le C\quad \text{if $I_0 \le \epsilon_0$},
\end{equation}
where the constant $C$ is independent of $T$. This contradicts \eqref{contra}. Therefore, $T_{\max}=\infty$. Moreover, by taking the limit of \eqref{est} as $T\to\infty$, we obtain \eqref{global-est}.
Thus, we conclude Theorem \ref{thm:sdge}. 
\end{proof}


\section{Examples of $A$}\label{sec:6}

\subsection{Dirichlet Laplacian on an open set}\label{sub:6.1}
Let $\Omega$ be an open set of $\mathbb R^d$ with $d\ge 1$, and let $A = -\Delta_D$ denote the Dirichlet Laplacian on $L^2(\Omega)$. 
Then $A$ has a non-negative and self-adjoint realization with the domain
\[
\mathrm{Dom}\, (A) 
=
\left\{
f \in H^1_0(\Omega) : \Delta f \in L^2(\Omega)\text{ in the distributional sense}
\right\},
\]
where $H^1_0(\Omega)$ is the completion of $C_0^\infty(\Omega)$, which is the set of all $C^\infty$-functions on $\Omega$ having compact support in $\Omega$, with respect to the Sobolev norm $\|\cdot\|_{H^1(\Omega)}$. Moreover, the integral kernel $p_D(t,x,y)$ of the semigroup $e^{-tA}$ generated by $A$ has the Gaussian upper bound
\[
0\le p_D(t,x,y) \le (4\pi t)^{-\frac{d}{2}} \exp\left(-\frac{|x-y|^2}{4t}\right)
\]
for any $t>0$ and almost everywhere $x,y\in\Omega$ 
(see, e.g., Propositions 2.1 and 3.1 in \cite{IMT-RMI}). 
This bound yields 
\[
\|e^{t\Delta_D} f\|_{L^\infty(\Omega)} \le \left\|e^{t\Delta} |\tilde{f}| \right\|_{L^\infty(\mathbb R^d)} 
\le C t^{-\frac{d}{4}}\|\tilde f\|_{L^2(\mathbb R^d)}
 = C t^{-\frac{d}{4}}\|f\|_{L^2(\Omega)}
\]
for any $f \in L^2(\Omega)$, where $\tilde{f}$ is the zero extension of $f$ to $\mathbb R^d$.
Thus, we see that the Dirichlet Laplacian $A = -\Delta_D$ satisfies Assumption \ref{assum:A} with $\alpha=d/4$. 
Similarly, $A = -\Delta_D$ also satisfies Assumption \ref{assum:B} with $\alpha=d/4$.
Then the exponent $p_F$ is 
\[
p_F = p_F\left(\frac{d}{4},q\right) = 1 + \frac{2q}{d},
\]  
and we have the following result as a corollary of Theorem \ref{thm:sdge}.

\begin{cor}\label{cor:sdge1}
Let $d \ge 1$, and let $q$ and $p$ satisfy 
\begin{align*}
	p > 1+\frac{2}{d}, \quad 
	\frac{2}{p} < \frac{d}{2}(p-1), \quad 
	\text{and} \quad 
	q \in
	\left[ 
	\max \left\{ 1, \frac{2}{p} \right\},
	\min \left\{ 2, \frac{d}{2}(p-1) \right\}
	\right].
\end{align*}
In addition, assume that 
\[
p \le \frac{d}{d-2}\quad \text{if }d\ge3.
\]
Then, the same statement as Theorem \ref{thm:sdge} holds with $A=-\Delta_D$. 
\end{cor}

We note that, when $d = 1, 2$,
the above conditions are simply written 
\begin{align*}
	p > 1+ \frac{2}{d} \quad \text{and} \quad
	q \in \left[ 1, \min\left\{ 2, \frac{d}{2}(p-1) \right\} \right].
\end{align*}

When
$\Omega = \mathbb{R}^d$
and $A = - \Delta$,
Corollary \ref{cor:sdge1}
slightly improves Theorems 1 and 2 of \cite{IkeOht-2002},
in which the assumption
\begin{align*}
	\begin{cases}
	p > 1+ \frac{2}{d} \quad \text{and} \quad
	q \in [1,2] \cap \left[ 1, \frac{d}{2}(p-1) \right).
	&(d=1,2),\\
	p > 1+\frac{2}{d}, \  
	\frac{2}{p} < \frac{d}{2}(p-1), \  
	\text{and} \  
	q \in
	\left[ 
	\frac{\sqrt{d^2+16d}-d}{4},
	\min \left\{ 2, \frac{d}{2}(p-1) \right\}
	\right)
	&(3 \le d \le 6)
	\end{cases}
\end{align*}
is imposed.
Compared to this, Corollary \ref{cor:sdge1}
relaxes the range of $q$ and, in particular,
includes the critical case
$q = \frac{d}{2}(p-1)$, i.e., $p = p_F(\frac{d}{4},q)$
except for the case $q=1$.
For the recent progress of the study of semilinear damped wave equation in
$\mathbb{R}^d$, 
we refer to the introduction and the references of
\cite{IIOW-2019}.

As for the Dirichlet problems of damped wave equations on exterior domains,
Ono \cite{Ono-2003}
proved the small data global existence under the assumption
\begin{align*}
	1 \le q \le \frac{2d}{d+2}
	\quad \text{and} \quad
	 \begin{cases}
	 1 + \frac{4}{d+2} < p < \infty &(d=2),\\
	 1 + \frac{4}{d+2} < p \le \frac{d}{d-2} &(3 \le d \le 5),
	 \end{cases}
\end{align*}
or $d=3, q=1$, $1+\frac{2}{d} < p \le 3$.
This result is included in Corollary \ref{cor:sdge1} in the case $d=2$,
though there is no mutual implication in the case $3 \le d \le 5$.
For the initial data in weighted Sobolev spaces,
we refer the reader to 
Sobajima \cite{Sob-2019}.

\subsection{Robin Laplacian on an exterior domain}\label{sub:6.2}
Let $d\ge3$ and $\Omega$ be an exterior domain in $\mathbb R^d$ of a compact and connected set with Lipschitz boundary.
We consider the Laplace operator $-\Delta_\gamma$ on $L^2(\Omega)$ associated with a quadratic form
\[
q_\gamma(f,g) = \int_\Omega \nabla f \cdot \overline{\nabla g}\,dx
+ \int_{\partial \Omega} \gamma f \overline{g}\,dS
\]
for any $f,g \in H^1(\Omega)$, where $\gamma$ is a function $\partial \Omega \to \mathbb R$ and $\partial \Omega$ denotes the boundary of $\Omega$.
Note that the case of $\gamma=0$ is the Neumann Laplacian $-\Delta_N$ on $L^2(\Omega)$.
Assume that $\gamma \in L^\infty(\partial \Omega)$ and $\gamma\ge0$.
We denote by $p_\gamma(t,x,y)$ and $p_N(t,x,y)$ the integral kernels of the semigroup generated by $-\Delta_\gamma$ and $-\Delta_N$, respectively. 
Then $p_\gamma(t,x,y)$ and $p_N(t,x,y)$ have the Gaussian upper bounds. In fact, it is known that 
\[
p_D(t,x,y) \le p_\gamma(t,x,y)\le p_N(t,x,y)
\lesssim t^{-\frac{d}{2}} \exp\left(-\frac{|x-y|^2}{Ct}\right)
\]
for any $t>0$ and almost everywhere $x,y\in\Omega$ and for some $C>0$. 
The first and second inequalities follow from domination of semigroups (see, e.g., Theorem 2.24 in \cite{Ouh_2005}),  
and the proof of the last inequality can be found in 
Chen, Williams and Zhao \cite{CWZ-1994}. 

In the low dimensional cases $d=1,2$, 
the Gaussian upper bounds are also obtained by Kova\v{r}\'{i}k and Mugnolo \cite{KM-2018}. 
More precisely, in the case $d=2$, if 
$\Omega$ is an exterior domain in $\mathbb R^2$ of a compact and connected set with $C^2$-boundary and
\[
\underset{x\in\partial \Omega}{\mathrm{ess\ inf}}\,
\gamma(x) >0, 
\]
then the Gaussian upper bound
\begin{equation}\label{GUB-R}
|p_\gamma(t,x,y)|\lesssim t^{-\frac{d}{2}} \exp\left(-\frac{|x-y|^2}{Ct}\right)
\end{equation}
holds for any $t>0$ and almost everywhere $x,y\in\Omega$ and for some $C>0$ (see Section~2 in \cite{KM-2018}).
In the case $d=1$, let $\Omega=\mathbb R_+$ and $-\Delta_\gamma$
is the Laplace operator on $L^2(\mathbb R_+)$ associated with a quadratic form
\[
q_\gamma(f,g) = \int_0^\infty f' \overline{g'}\,dx
+ \gamma f(0)\overline{g(0)}
\]
for any $f,g \in H^1(\mathbb R_+)$, where $\gamma\ge 0$ is a constant.
Then we also have the Gaussian upper bound \eqref{GUB-R} (see Section 4 in \cite{KM-2018}).
Thus, the Robin Laplacian $A = -\Delta_\gamma$ satisfies Assumptions \ref{assum:A} and \ref{assum:B} with $\alpha=d/4$.  
In these cases, we have the same result as Corollary \ref{cor:sdge1} with $A = -\Delta_\gamma$.

\subsection{Schr\"odinger operator with a Kato type potential}\label{sub:6.3}
Let $\Omega$ be an open set of $\mathbb R^d$ with $d\ge 1$. We consider 
the Schr\"odinger operator $A= -\Delta + V$ with the homogeneous Dirichlet boundary condition on $\Omega$, where 
$V=V(x)$ is a real-valued measurable function on $\Omega$ such that 
\[
V=V_+ - V_-,\quad V_{\pm} \ge0,\quad V_+ \in L^1_{\mathrm{loc}}(\Omega)\quad \text{and}\quad
V_- \in K_d(\Omega).
\]
We say that $V_-$ belongs to the Kato class $K_{d}(\Omega)$ if
\begin{align}\notag
\left\{
\begin{aligned}
	&\lim_{r \rightarrow 0} \sup_{x \in \Omega} \int_{\Omega \cap \{|x-y|<r\}}
	   \frac{V_-(y)}{|x-y|^{d-2}} \,dy = 0 &\text{for }d\ge 3, \\
	&\lim_{r \rightarrow 0} \sup_{x \in \Omega} \int_{\Omega \cap \{|x-y|<r\}}
	   \log (|x-y|^{-1})V_-(y) \,dy = 0 &\text{for }d=2, \\
	&\sup_{x \in \Omega}\int_{\Omega \cap \{|x-y|<1\}} V_-(y) \,dy <\infty &\text{for }d=1\,
	\end{aligned}\right.
\end{align}
(see Section A.2 in Simon \cite{Simon-1982}).
It is readily seen that the potential $V_-(x) = 1/|x|^\alpha$ with $0\le \alpha <2$ if $d\ge2$ and $0\le \alpha <1$ if $d=1$
is included in $K_{d}(\Omega)$. 
It should be noted that the potential like $V_-(x) = 1 / |x|^2$ near $x=0$ 
is excluded from $K_d(\Omega)$.
In addition, we assume that the negative part $V_-$ satisfies
\[
\begin{cases}
\displaystyle\sup_{x \in \Omega} \int_{\Omega} \frac{V_-(y)}{|x-y|^{d-2}} \,dy
 < \dfrac{\pi^{\frac{d}{2}}}{\Gamma (\frac{d}{2}-1)}
& \quad \text{if } d \geq 3,\\
V _- = 0
& \quad \text{if } d = 1, 2.
\end{cases}
\]
Then $A$ has a non-negative and self-adjoint realization on $L^2(\Omega)$, and the integral kernel of its semigroup satisfies the same bound as \eqref{GUB-R} 
(see, e.g., Propositions 2.1 and 3.1 in \cite{IMT-RMI}). 
Thus, the Schr\"odinger operator $A$ satisfies Assumptions \ref{assum:A} and \ref{assum:B} with $\alpha=d/4$, and the same result as Corollary \ref{cor:sdge1} holds with $A= -\Delta + V$.

\subsection{Schr\"odinger operator with a Dirac delta potential}\label{sub:6.4}
Let $X=\mathbb R$. We consider the Schr\"odinger operator $A = -(1/2)\partial_x^2 + q\delta (x)$ with repulsive delta potential $q\delta$, $q>0$. 
Then $A$ has a non-negative and self-adjoint realization on $L^2(\mathbb R)$ (see Theorem 3.1.1 in Chapter 1.3 in \cite{AGHH_1988} for the details) and the estimates
\begin{equation}\label{decay-est-d}
\|e^{-tA}\|_{L^{q_1}\to L^{q_2}} \lesssim t^{-\frac{1}{2}(\frac{1}{q_1}-\frac{1}{q_2})}
\end{equation}
for any $t>0$ and $1\le q_1 \le q_2 \le \infty$. 
Thus, the Schr\"odinger operator $A$ satisfies Assumption \ref{assum:A} with $\alpha = 1/4$ (Assumption \ref{assum:B} is unnecessary as $\alpha=1/4$), and the same result as Corollary \ref{cor:sdge1} holds with $d=1$. 

In the rest of this subsection, let us give a sketch of proof of \eqref{decay-est-d}. 
The proof is based on the argument of proof of $L^1$-$L^\infty$ estimate (2.23) for Schr\"odinger group $e^{itA}$ in Segata \cite{Seg-2015}.
Let $\mathbf{1}_{+}$ and $\mathbf{1}_{-}$ be characteristic functions on $[0,\infty)$ and $(-\infty, 0)$, respectively. 
We have the representation formula for $e^{-tA}$:
\[
(e^{-tA} f) (x) 
=
\begin{cases}
\displaystyle C t^{-\frac12} \int_{-\infty}^\infty e^{-\frac{|x-y|^2}{2t}} \mathscr L_+[f](y)\, dy\quad &\text{if }x\ge0,\\
\displaystyle C t^{-\frac12} \int_{-\infty}^\infty e^{-\frac{|x-y|^2}{2t}} \mathscr L_-[f](y)\, dy\quad &\text{if }x<0,
\end{cases}
\]
where 
\[
\mathscr L_{\pm}[f](x) 
:=
f(x) - q \mathbf{1}_{\mp}(x)e^{\pm qx}\int_{\pm x}^{\mp x} e^{q|y|}f(y)\, dy,\quad x \in \mathbb R
\]
(see Lemma 2.1 in Holmer, Marzuola and Zworski \cite{HMZ-2007} and also Proposition 2.1 in \cite{Seg-2015}). 
Then, 
\[
\|e^{-tA} f \|_{L^{q_2}} \lesssim t^{-\frac{1}{2}(\frac{1}{q_1}-\frac{1}{q_2})} (\|\mathscr L_+[f]\|_{L^{q_1}} + \|\mathscr L_-[f]\|_{L^{q_1}}).
\]
Now, it is readily seen that $\mathscr L_{\pm}$ are linear and bounded on $L^1(\mathbb R)$ and $L^\infty(\mathbb R)$. 
Hence, by the Riesz-Thorin interpolation theorem, $\mathscr L_{\pm}$ are bounded on $L^q(\mathbb R)$ for any $1\le q \le \infty$.
Therefore, the estimates \eqref{decay-est-d} are obtained. 

\subsection{Elliptic operators}\label{sub:6.5}
Let $\Omega$ be an open set of $\mathbb R^d$ with $d\ge 1$.
We consider the self-adjoint operator $A$ associated with a quadratic form
\[
q(f,g) =
\int_{\Omega}
\left\{
\sum_{k,j=1}^d a_{kj}(x) \partial_{x_k} f(x) \overline{\partial_{x_j} g(x)} + 
a_0(x) f(x) \overline{g(x)}
\right\}\,dx
\]
for any $f, g \in H^1_0(\Omega)$, where
$a_{kj}, a_0 \in L^\infty(\Omega)$ are real-valued functions for all $1\le j,k\le d$, $a_0\ge0$, and the principle part is elliptic, i.e.,
there exists a constant $\eta>0$ such that
\[
\sum_{j,k=1}^d a_{kj}(x) \xi_j \overline{\xi_k} \ge \eta |\xi|^2,\quad \xi \in\mathbb C^{d}, \text{ a.e.}\, x \in \Omega.
\]
The integral kernel of $e^{-tA}$ has the same bound as \eqref{GUB-R}
(see, e.g., Theorem~6.8 in \cite{Ouh_2005}). 
Thus, the elliptic operator $A$ satisfies Assumptions \ref{assum:A} and \ref{assum:B} with $\alpha=d/4$, and the same result as Corollary \ref{cor:sdge1} holds with the elliptic operator $A$. 

\subsection{Laplace operator on Sierpinski gasket in $\mathbb R^d$}\label{sub:6.6}
Let $X$ be the unbounded Sierpinski gasket in $\mathbb R^d$ with $d\ge 2$. 
Let $\rho$ be the induced metric on $X$, and let $\mu$  be the Hausdorff measure on $X$ of dimension $D_{SG} = \log_2 (d + 1)$. 
The measure $\mu$ is $\sigma$-finite and it satisfies 
\[
\mu ( B(x,r)) \lesssim r^{D_{SG}}
\]
for any $x\in X$ and $r>0$, where $B(x,r) := \{y \in X : \rho(x,y)\le r \}$, 
and the standard local Dirichlet form on $X$ generates the non-negative self-adjoint operator $A$ on $L^2(X)$. Then the integral kernel 
$p_{SG}(t,x,y)$ of its semigroup satisfies 
\[
|p_{SG}(t,x,y)|
\lesssim 
t^{-\frac{D_{SG}}{m}} \exp \left(
- \frac{\rho(x,y)^{\frac{m}{m-1}}}{C t^{\frac{1}{m-1}}}
\right)
\]
for any $t>0$ and almost everywhere $x,y \in X$ and for some $C>0$, where $m=\log_2 (d+3)$ 
(see Section 2 and Theorem 8.18 in Barlow \cite{Bar_1995} and also Subsection 3.3 in Bui, D'Ancona and Nicola \cite{BDN-2020}). 
Thus, the operator $A$ satisfies Assumption \ref{assum:A} with $\alpha = D_{SG}/(2m) = (\log_2 (d + 1))/(2\log_2 (d + 3))$. 
Here, we note that $\alpha$ is always less than $1/2$ for any $d\ge2$, and hence, Assumption~\ref{assum:B} does not appear. 
The exponent $p_F$ is 
\[
p_F = p_F\left(\frac{D_{SG}}{2m},q\right) = 1 + \frac{mq}{D_{SG}} = 1 + \frac{\log_2 (d+3)}{\log_2 (d + 1)} q.
\]
Noting
$p_F\left(\frac{D_{SG}}{2m},1\right) >2$,
we have the following result as a corollary of Theorem \ref{thm:sdge}.

\begin{cor}\label{cor:sdge2}
Let $d \ge 2$, and let $q$ and $p$ satisfy 
\begin{align*}
	p > p_F\left(\frac{D_{SG}}{2m},1\right) \ 
	\text{and} \ 
	q \in
	\left[
	1,
	\min\left\{ 2, \frac{D_{SG}}{m}(p-1) \right\}
	\right].
\end{align*}
Then, the same statement as Theorem \ref{thm:sdge} holds with the Laplace operator $A$ on Sierpinski gasket $X$. 
\end{cor}

\begin{rem}
In this case, the additional assumption \eqref{additional-assum} does not appear in the corollary, as $\alpha < 1/2$. 
\end{rem}

\subsection{Fractional Laplacian}\label{sub:6.7}

Let $\Omega$ be an open set of $\mathbb R^d$, and let $A$ be a non-negative and self-adjoint operator on $L^2(\Omega)$. 
Assume that there exists an integral kernel $p(t,x,y)$ of $e^{-tA}$ and it satisfies 
\[
|p(t,x,y)|
\lesssim 
t^{-\frac{d}{m}} \exp \left(
- \frac{|x-y|^{\frac{m}{m-1}}}{C t^{\frac{1}{m-1}}}
\right)
\]
for any $t>0$ and almost everywhere $x,y \in \Omega$ and for some $C>0$, where $m>1$. 
Then, $A$ has no zero eigenvalues, and 
the operators $\phi_j(\sqrt{A})$ are uniformly bounded on $L^p(\Omega)$ with respect to $j \in \mathbb Z$ (see the proof of Theorem 1.4 in \cite{BDN-2020}),
where $\{\phi_j\}_j$ is the Littlewood-Paley decomposition.
Therefore, for $s \in \mathbb R$ and $1\le r \le \infty$, we can define the Besov spaces $\dot B^s_{q,r}(A)$ and Triebel-Lizorkin spaces $\dot F^s_{q,r}(A)$ associated with $A$ by 
\[
\|f\|_{\dot B^s_{q,r}(A)} := 
\left\{\sum_{j\in \mathbb Z} 2^{srj}\left\|\phi_j(\sqrt{A})f\right\|_{L^q(\Omega)}^r\right\}^{\frac{1}{r}}\quad \text{for }1\le q \le \infty,
\]
\[
\|f\|_{\dot F^s_{q,r}(A)} := \left\|
\left\{\sum_{j\in \mathbb Z} \left|2^{sj}\phi_j(\sqrt{A})f\right|^r\right\}^{\frac{1}{r}}
\right\|_{L^q(\Omega)} \quad \text{for }1\le q < \infty
\]
(with the usual modifications for $r=\infty$), 
and the fundamental properties of $\dot B^s_{q,r}(A)$ and $\dot F^s_{q,r}(A)$ can be also shown, such as completeness, embedding relations and duality.
In particular, we have the relations 
\begin{equation}\label{relations-besov}
\dot F^0_{q, 2}(A) = L^q(\Omega),
\quad L^q(\Omega) \hookrightarrow 
\dot B^0_{q, \infty}(A)\quad 
\text{and}\quad 
\dot B^0_{\infty, 1}(A) \hookrightarrow L^\infty(\Omega)
\end{equation} 
for any $1 < q < \infty$.
For the details of $\dot B^s_{q,r}(A)$ and $\dot F^s_{q,r}(A)$, we refer to \cite{BBD-2020,IMT-Besov}\footnote{In \cite{BBD-2020}, the homogeneous Triebel-Lizorkin spaces are studied on metric measure spaces with the doubling condition via the spectral approach. On the other hand, in \cite{IMT-Besov}, the theory of homogeneous Besov spaces associated with $-\Delta_D$ on an open set is established. In combination of these arguments, the theory of $\dot B^s_{q,r}(A)$ and $\dot F^s_{q,r}(A)$ can be established.}. 
Based on the argument of the proof of Theorem 1.1 in Iwabuchi \cite{Iwa-2018}, we can obtain 
the following estimates for $e^{-tA^\frac{\nu}{2}} $ with $\nu>0$:
\[
\| e^{-tA^\frac{\nu}{2}} \|_{\dot F^0_{q,2}(A) \to \dot F^0_{q,2}(A)} \lesssim
1
\quad \text{and}\quad 
\| e^{-tA^\frac{\nu}{2}} \|_{\dot B^0_{q,\infty}(A) \to \dot B^0_{\infty, 1}(A)} \lesssim
t^{-\frac{2d}{m \nu q}}
\]
for any $t>0$ and $1 \le q < \infty$. By combining these estimates with the relations \eqref{relations-besov}, we find that
\[
\| e^{-tA^\frac{\nu}{2}} \|_{L^q \to L^q} \lesssim
1
\quad \text{and}\quad 
\| e^{-tA^\frac{\nu}{2}} \|_{L^q \to L^\infty} \lesssim
t^{-\frac{2d}{m\nu q}}
\]
for any $t>0$ and $1 < q < \infty$. 
Thus, the operator $A^{\nu/2}$ satisfies Assumptions \ref{assum:A} and \ref{assum:B} with $\alpha=d/(m \nu)$.
Hence, 
\[
p_F = p_F\left(\frac{d}{m \nu},q\right) = 1 + \frac{m\nu q}{2d},
\] 
and we have the following result as a corollary of Theorem \ref{thm:sdge}.

\begin{cor}\label{cor:sdge3}
Let $d \ge 1$, and let $q$ and $p$ satisfy 
\begin{align*}
	p > p_F\left(\frac{d}{m \nu},1\right), \ 
	\frac{2}{p} < \frac{2d}{m \nu} (p-1), \ 
	\text{and} \ 
	q \in 
	\left[
	\max\left\{ 1, \frac{2}{p} \right\},
	\min\left\{ 2, \frac{2d}{m \nu} (p-1) \right\}
	\right].
\end{align*}
In addition, assume that 
\[
p \le \frac{2d}{2d-m\nu}\quad \text{if }d > \frac{m\nu}{2}.
\]
Then, the same statement as Theorem \ref{thm:sdge} holds with $A^\frac{\nu}{2}$. 
\end{cor}

In the case $\Omega = \mathbb R^d$ and $A= (-\Delta)^\frac{\nu}{2}$, 
the damped wave equations have been studied (see \cite{FIW-2020,Tak-2011} and references therein). 
In particular, the corresponding results were proved in the weighted $L^2$-based Sobolev spaces, instead of $L^q(\mathbb R^d)$ (see Theorem~2.7 in \cite{FIW-2020}). 

\subsection{Other examples}\label{sub:6.8}
Other examples are the Laplace-Beltrami operators on Riemann manifolds, the Laplace operators on homogeneous groups, the sub-Laplacian operators on Heisenberg groups, etc.
We here omit the details, and we refer to Section~3 in \cite{BDN-2020} for instance.

\appendix

\section{}

In this appendix we give the Sobolev inequality and the Gagliardo-Nirenberg inequality associated with $A$, which are used in this paper.

Let $(X,\mu)$ be a $\sigma$-finite measure space, and let $A$ be a self-adjoint operator on $L^2(X)$. 
For a Borel measurable function $\phi$ on $\mathbb R$, 
an operator $\phi(A)$ is defined by the spectral decomposition 
\[
\phi(A) = \int_{\sigma(A)} \phi(\lambda)\, dE_{A}(\lambda)
\]
with the domain 
\[
\mathrm{Dom} (\phi(A)) = 
\left\{
f \in L^2(X) : \int_{\sigma(A)} |\phi(\lambda)|^2\, d\|E_A(\lambda) f\|_{L^2}^2 < \infty
\right\},
\]
where $\sigma(A)$ denotes the spectrum of $A$ and $\{E_{A}(\lambda)\}_ {\lambda\in\mathbb R}$ is the spectral resolution of the identity for $A$. 
For $s\ge0$, 
the Sobolev space $H^s(A)$ is defined by 
\[
H^s(A) := \{ f\in L^2(X) : \|f\|_{H^s(A)}<\infty\}
\]
with the norm
\[
\|f\|_{H^s(A)} := \|(I+A)^{\frac{s}{2}} f\|_{L^2(X)}.
\]
The space $H^{-s}(A)$ denotes the dual space of $H^s(A)$. For simplicity, we may assume $A$ is non-negative on $L^2(X)$. Then,
by a combination of the spectral decomposition of $A$ with the Laplace transform, we have 
the formulas
\begin{equation}\label{appA:formula1}
(I+A)^{-\frac{s_0}{2}}
=
\frac{1}{\Gamma(\frac{s_0}{2})} \int_{0}^\infty
t^{\frac{s_0}{2}-1} e^{-t} e^{-tA}\, dt
\quad \text{with $s_0>0$},
\end{equation}
\begin{equation}\label{appA:formula2}
A^{-\frac12} 
= \frac{1}{\Gamma\left(\frac12\right)} 
\int_0^\infty t^{-\frac12} e^{-tA} \, dt.
\end{equation}

Under Assumption \ref{assum:A}, we have $L^{q_1}$-$L^{q_2}$ estimates
\begin{equation}\label{A:LpLq_1}
\|A^{\frac{s}{2}}e^{-tA}\|_{L^{q_1}\to L^{q_2}} \lesssim t^{-2\alpha(\frac{1}{q_1}-\frac{1}{q_2}) -\frac{s}{2}}
\end{equation}
for any $t>0$ and for any $s\ge0$ and $1\le q_1 \le 2 \le q_2 \le \infty$. 
The proof of \eqref{A:LpLq_1} is based on the Riesz-Thorin interpolation theorem together with Assumption \ref{assum:A} and the duality argument
 (see the proof of \eqref{LpLq_1} in Lemma \ref{lem:LpLq}). 
Then we obtain the Sobolev inequality and the Gagliardo-Nirenberg inequality associated with $A$.

\begin{lem}[Sobolev inequality]\label{lem:Sobolev}
Suppose that Assumption \ref{assum:A} holds. 
Let $q \in [2,\infty]$ and $s>2\alpha(1-2/q)$. Then 
\[
\|f\|_{L^q} \lesssim  \|f\|_{H^s(A)}
\]
for any $f\in H^s(A)$. 
\end{lem}

\begin{lem}[Gagliardo-Nirenberg inequality]\label{lem:GN}
Suppose that Assumption \ref{assum:A} holds. 
Let $q \in (2,\infty]$ with $\alpha ( 1/2 - 1/q ) < 1/4$.
Then 
\begin{equation}\label{GN}
\|f\|_{L^q} \lesssim  \|f\|_{L^2}^{-4\alpha(\frac12 - \frac{1}{q}) +1} \|A^{\frac12}f\|_{L^2}^{4\alpha(\frac12 - \frac{1}{q})}
\end{equation}
for any $f \in H^1(A)$. 
\end{lem}

Imposing additionally Assumption \ref{assum:B}, 
we also obtain the Gagliardo-Nirenberg inequality \eqref{GN} with $\alpha ( 1/2 - 1/q ) = 1/4$, i.e., the critical Sobolev inequality.

\begin{lem}[Critical Sobolev inequality]\label{lem:cS}
Suppose that Assumptions \ref{assum:A} and \ref{assum:B} hold. 
Let $\alpha > 1/2$. Then 
\begin{equation}\label{c-Sobolev}
\|f\|_{L^\frac{4\alpha}{2\alpha-1}} \lesssim  \|A^{\frac12}f\|_{L^2}
\end{equation}
for any $f\in H^1(A)$.
\end{lem}

The proof of Lemma \ref{lem:Sobolev} is the same as the proof of \eqref{Sobolev-special} with $L^2$-$L^q$ estimates \eqref{A:LpLq_1} instead of $L^2$-$L^\infty$ estimates.
So, we may omit the proof. 

The proofs of Lemmas \ref{lem:GN} and \ref{lem:cS}
are the same as those of Theorem~6.2 in Ouhabaz \cite{Ouh_2005} and Theorem 2.4.2 in Davies \cite{D_1989}, respectively. 
Here, we give their proofs to make the paper self-contained. 

\begin{proof}[Proof of Lemma \ref{lem:GN}]
For $f \in H^1(A)$, we write 
\[
f = e^{-tA}f + \int_0^t Ae^{-sA}f\, ds.
\]
Then we see from \eqref{A:LpLq_1} that  
\[
\begin{split}
\|f\|_{L^q} 
& \le \|e^{-tA}f\|_{L^q} + \int_0^t \|A^{\frac12}e^{-sA}A^{\frac12}f\|_{L^q}\, ds\\
& \lesssim 
t^{-2\alpha(\frac12 - \frac{1}{q})}\|f\|_{L^2} + 
\int_0^t s^{-2\alpha(\frac12 - \frac{1}{q})-\frac12}\, ds \cdot \|A^{\frac12}f\|_{L^2}\\
& \lesssim 
t^{-2\alpha(\frac12 - \frac{1}{q})}\|f\|_{L^2} + 
t^{-2\alpha(\frac12 - \frac{1}{q})+\frac12} \|A^{\frac12}f\|_{L^2},
\end{split}
\]
since $\alpha ( 1/2 - 1/q ) < 1/4$. 
Therefore, by taking $t = \|f\|_{L^2}^2 \|A^{1/2}f\|_{L^2}^{-2}$, we obtain \eqref{GN}. 
The proof of Lemma \ref{lem:GN} is finished.
\end{proof}

\begin{proof}[Proof of Lemma \ref{lem:cS}]
Let $q_1 \in [1,2)$ and $q_2 \in (2,4\alpha)$ as in Assumption \ref{assum:B}. We use \eqref{appA:formula2} to write
\[
A^{-\frac12} f 
= \frac{1}{\Gamma\left(\frac12\right)} \left(
\int_0^T + \int_T^\infty
\right) t^{-\frac12} e^{-tA} f \, dt
= : f_1 + f_2
\]
for $T>0$. 
We see from the second estimate in Assumption \ref{assum:B} that, for $j=1,2$,
\[
\|f_2\|_{L^\infty}
\le C_0 T^{\frac12 -\frac{2\alpha}{q_j}} \|f\|_{L^{q_j}}
\]
with some $C_0>0$. 
Given $\eta>0$, we take $T$ to satisfy 
$
\eta/2 = C_0 T^{\frac12 -\frac{2\alpha}{q_j}} \|f\|_{L^{q_j}}
$.
We deduce from the first estimate in Assumption \ref{assum:B} that 
\[
\begin{split}
\big|\big\{ x : | A^{-\frac12}f(x)| \ge \eta\big\}\big|
& \le 
\bigg|\bigg\{ x : | f_1(x)| \ge \frac{\eta}{2}\bigg\}\bigg|
\le \bigg(\frac{\eta}{2}\bigg)^{-q_j}\|f_1\|_{L^{q_j}}^{q_j}\\
& \le 
C
\bigg(\frac{\eta}{2}\bigg)^{-{q_j}}
T^\frac{q_j}{2} \|f\|_{L^{q_j}}^{q_j}
\le 
C \eta^{-r_j}
 \|f\|_{L^{q_j}}^{r_j},
\end{split}
\]
where $r_j\in (1,\infty)$ satisfies $1/r_j = 1/q_j - 1/4\alpha$.
This implies that $A^{-\frac12}$ is of weak type $(q_j,r_j)$. 
By applying the Marcinkiewicz interpolation theorem, we deduce that $A^{-\frac12}$ is bounded from $L^2$ into $L^\frac{4\alpha}{2\alpha-1}$, which is equivalent to \eqref{c-Sobolev}.
The proof of Lemma \ref{lem:cS} is finished.
\end{proof}

\vspace{5mm}

\noindent{\bf Acknowledgment.} 
This work was supported by JSPS KAKENHI Grant Numbers 
JP16K17625,
JP18H01132,
JP19K14581,
JP19J00206,
JP20K14346,
and  JST CREST
Grant Number JPMJCR1913.


\begin{bibdiv}
\begin{biblist}


\bib{AGHH_1988}{book}{
   author={Albeverio, Sergio},
   author={Gesztesy, Friedrich},
   author={H\o egh-Krohn, Raphael},
   author={Holden, Helge},
   title={Solvable models in quantum mechanics},
   series={Texts and Monographs in Physics},
   publisher={Springer-Verlag, New York},
   date={1988},
}

\bib{Bar_1995}{article}{
   author={Barlow, Martin T.},
   title={Diffusions on fractals},
   conference={
      title={Lectures on probability theory and statistics},
      address={Saint-Flour},
      date={1995},
   },
   book={
      series={Lecture Notes in Math.},
      volume={1690},
      publisher={Springer, Berlin},
   },
   date={1998},
   pages={1--121},
}

\bib{BBD-2020}{article}{
   author={Bui, Huy-Qui},
   author={Bui, The Anh},
   author={Duong, Xuan Thinh},
   title={Weighted Besov and Triebel-Lizorkin spaces associated with
   operators and applications},
   journal={Forum Math. Sigma},
   volume={8},
   date={2020},
   pages={Paper No. e11, 95},
}


\bib{BDN-2020}{article}{
   author={Bui, The Anh},
   author={D'Ancona, Piero},
   author={Nicola, Fabio},
   title={Sharp $L^p$ estimates for Schr\"{o}dinger groups on spaces of
   homogeneous type},
   journal={Rev. Mat. Iberoam.},
   volume={36},
   date={2020},
   number={2},
   pages={455--484},
}

\bib{CWZ-1994}{article}{
   author={Chen, Z. Q.},
   author={Williams, R. J.},
   author={Zhao, Z.},
   title={A Sobolev inequality and Neumann heat kernel estimate for
   unbounded domains},
   journal={Math. Res. Lett.},
   volume={1},
   date={1994},
   number={2},
   pages={177--184},
}

\bib{ChiHar-2003}{article}{
   author={Chill, Ralph},
   author={Haraux, Alain},
   title={An optimal estimate for the difference of solutions of two
   abstract evolution equations},
   journal={J. Differential Equations},
   volume={193},
   date={2003},
   number={2},
   pages={385--395},
}


\bib{D_1989}{book}{
   author={Davies, E. B.},
   title={Heat kernels and spectral theory},
   series={Cambridge Tracts in Mathematics},
   volume={92},
   publisher={Cambridge University Press, Cambridge},
   date={1989},
}

\bib{Fujita}{article}{
   author={Fujita, Hiroshi},
   title={On the blowing up of solutions of the Cauchy problem for
   $u_{t}=\Delta u+u^{1+\alpha }$},
   journal={J. Fac. Sci. Univ. Tokyo Sect. I},
   volume={13},
   date={1966},
   pages={109--124 (1966)},
}

\bib{FIW-2020}{article}{
   author={Fujiwara, Kazumasa},
   author={Ikeda, Masahiro},
   author={Wakasugi, Yuta},
   title={The Cauchy problem of the semilinear second order evolution equation with fractional Laplacian and damping},
   journal={to appear in NoDEA, Nonlinear Differ. Equ. Appl., arXiv:2003.09239},
   date={2020},
}


\bib{GP-2020}{article}{
   author={Georgiev, Vladimir},
   author={Palmieri, Alessandro},
   title={Critical exponent of Fujita-type for the semilinear damped wave
   equation on the Heisenberg group with power nonlinearity},
   journal={J. Differential Equations},
   volume={269},
   date={2020},
   number={1},
   pages={420--448},
}

\bib{HKN-2004}{article}{
   author={Hayashi, Nakao},
   author={Kaikina, Elena I.},
   author={Naumkin, Pavel I.},
   title={Damped wave equation with super critical semilinearities},
   journal={Differential Integral Equations},
   volume={17},
   date={2004},
   number={5-6},
   pages={637--652},
}

\bib{HMZ-2007}{article}{
   author={Holmer, Justin},
   author={Marzuola, Jeremy},
   author={Zworski, Maciej},
   title={Fast soliton scattering by delta impurities},
   journal={Comm. Math. Phys.},
   volume={274},
   date={2007},
   number={1},
   pages={187--216},
}

\bib{HO-2004}{article}{
   author={Hosono, Takafumi},
   author={Ogawa, Takayoshi},
   title={Large time behavior and $L^p$-$L^q$ estimate of solutions of
   2-dimensional semilinear damped wave equations},
   journal={J. Differential Equations},
   volume={203},
   date={2004},
   number={1},
   pages={82--118},
}

\bib{IIOW-2019}{article}{
   author={Ikeda, Masahiro},
   author={Inui, Takahisa},
   author={Okamoto, Mamoru},
   author={Wakasugi, Yuta},
   title={$L^p$-$L^q$ estimates for the damped wave equatioan and the
   critical exponent for the semilinear problem with slowly decaying data},
   journal={Commun. Pure Appl. Anal.},
   volume={18},
   date={2019},
   number={4},
   pages={1967--2008},
}

\bib{IS-2019}{article}{
   author={Ikeda, Masahiro},
   author={Sobajima, Motohiro},
   title={Sharp upper bound for lifespan of solutions to some critical
   semilinear parabolic, dispersive and hyperbolic equations via a test
   function method},
   journal={Nonlinear Anal.},
   volume={182},
   date={2019},
   pages={57--74},
}

\bib{Ike-2002}{article}{
   author={Ikehata, Ryo},
   title={Diffusion phenomenon for linear dissipative wave equations in an
   exterior domain},
   journal={J. Differential Equations},
   volume={186},
   date={2002},
   number={2},
   pages={633--651},
}

\bib{Ike-2003}{article}{
   author={Ikehata, Ryo},
   title={Fast decay of solutions for linear wave equations with dissipation
   localized near infinity in an exterior domain},
   journal={J. Differential Equations},
   volume={188},
   date={2003},
   number={2},
   pages={390--405},
}

\bib{IkeNis-2003}{article}{
   author={Ikehata, Ryo},
   author={Nishihara, Kenji},
   title={Diffusion phenomenon for second order linear evolution equations},
   journal={Studia Math.},
   volume={158},
   date={2003},
   number={2},
   pages={153--161},
}


\bib{IkeOht-2002}{article}{
   author={Ikehata, Ryo},
   author={Ohta, Masahito},
   title={Critical exponents for semilinear dissipative wave equations in
   $\bold R^N$},
   journal={J. Math. Anal. Appl.},
   volume={269},
   date={2002},
   number={1},
   pages={87--97},
}

\bib{IkeTan-2005}{article}{
   author={Ikehata, Ryo},
   author={Tanizawa, Kensuke},
   title={Global existence of solutions for semilinear damped wave equations
   in $\bold R^N$ with noncompactly supported initial data},
   journal={Nonlinear Anal.},
   volume={61},
   date={2005},
   number={7},
   pages={1189--1208},
}

%

\bib{Iwa-2018}{article}{
   author={Iwabuchi, Tsukasa},
   title={The semigroup generated by the Dirichlet Laplacian of fractional
   order},
   journal={Anal. PDE},
   volume={11},
   date={2018},
   number={3},
   pages={683--703},
}

\bib{IMT-RMI}{article}{
   author={Iwabuchi, Tsukasa},
   author={Matsuyama, Tokio},
   author={Taniguchi, Koichi},
   title={Boundedness of spectral multipliers for Schr\"odinger operators on open sets},
   journal={Rev. Mat. Iberoam.}
   volume={34},
   date={2018},
   number={3},
   pages={1277--1322},
}

\bib{IMT-Besov}{article}{
   author={Iwabuchi, Tsukasa},
   author={Matsuyama, Tokio},
   author={Taniguchi, Koichi},
   title={Besov spaces on open sets},
   journal={Bull. Sci. Math.},
   volume={152},
   date={2019},
   pages={93--149},
}

\bib{KM-2018}{article}{
   author={Kova\v{r}\'{i}k, Hynek},
   author={Mugnolo, Delio},
   title={Heat kernel estimates for Schr\"{o}dinger operators on exterior
   domains with Robin boundary conditions},
   journal={Potential Anal.},
   volume={48},
   date={2018},
   number={2},
   pages={159--180},
}

\bib{Mat-1976}{article}{
   author={Matsumura, Akitaka},
   title={On the asymptotic behavior of solutions of semi-linear wave
   equations},
   journal={Publ. Res. Inst. Math. Sci.},
   volume={12},
   date={1976/77},
   number={1},
   pages={169--189},
}

\bib{Nar-2004}{article}{
   author={Narazaki, Takashi},
   title={$L^p$-$L^q$ estimates for damped wave equations and their
   applications to semi-linear problem},
   journal={J. Math. Soc. Japan},
   volume={56},
   date={2004},
   number={2},
   pages={585--626},
}

\bib{Nis-2003}{article}{
   author={Nishihara, Kenji},
   title={$L^p$-$L^q$ estimates of solutions to the damped wave equation in
   3-dimensional space and their application},
   journal={Math. Z.},
   volume={244},
   date={2003},
   number={3},
   pages={631--649},
}

\bib{Nishiyama-2016}{article}{
   author={Nishiyama, Hisashi},
   title={Remarks on the asymptotic behavior of the solution to damped wave
   equations},
   journal={J. Differential Equations},
   volume={261},
   date={2016},
   number={7},
   pages={3893--3940},
}

\bib{Ono-2003}{article}{
   author={Ono, Kosuke},
   title={Decay estimates for dissipative wave equations in exterior
   domains},
   journal={J. Math. Anal. Appl.},
   volume={286},
   date={2003},
   number={2},
   pages={540--562},
}

\bib{Ouh_2005}{book}{
   author={Ouhabaz, El Maati},
   title={Analysis of heat equations on domains},
   series={London Mathematical Society Monographs Series},
   volume={31},
   publisher={Princeton University Press, Princeton, NJ},
   date={2005},
}

\bib{Pal-2020}{article}{
   author={Palmieri, Alessandro},
   title={Decay estimates for the linear damped wave equation on the
   Heisenberg group},
   journal={J. Funct. Anal.},
   volume={279},
   date={2020},
   number={9},
   pages={108721, 23},
}

\bib{Pal-2021}{article}{
   author={Palmieri, Alessandro},
   title={On the blow-up of solutions to semilinear damped wave equations
   with power nonlinearity in compact Lie groups},
   journal={J. Differential Equations},
   volume={281},
   date={2021},
   pages={85--104},
}

\bib{RTY-2011}{article}{
   author={Radu, Petronela},
   author={Todorova, Grozdena},
   author={Yordanov, Borislav},
   title={Diffusion phenomenon in Hilbert spaces and applications},
   journal={J. Differential Equations},
   volume={250},
   date={2011},
   number={11},
   pages={4200--4218},
}

\bib{RTY-2016}{article}{
   author={Radu, Petronela},
   author={Todorova, Grozdena},
   author={Yordanov, Borislav},
   title={The generalized diffusion phenomenon and applications},
   journal={SIAM J. Math. Anal.},
   volume={48},
   date={2016},
   number={1},
   pages={174--203},
}

\bib{RT-2019}{article}{
   author={Ruzhansky, Michael},
   author={Tokmagambetov, Niyaz},
   title={On semilinear damped wave equations for positive operators. I.
   Discrete spectrum},
   journal={Differential Integral Equations},
   volume={32},
   date={2019},
   number={7-8},
   pages={455--478},
}

\bib{Seg-2015}{article}{
   author={Segata, Jun-Ichi},
   title={Final state problem for the cubic semilinear Schr\"{o}dinger equation
   with repulsive delta potential},
   journal={Comm. Partial Differential Equations},
   volume={40},
   date={2015},
   number={2},
   pages={309--328},
}

\bib{Simon-1982}{article}{
   author={Simon, Barry},
   title={Schr\"{o}dinger semigroups},
   journal={Bull. Amer. Math. Soc. (N.S.)},
   volume={7},
   date={1982},
   number={3},
   pages={447--526},
}

\bib{Sob-2019}{article}{
 author={Sobajima, Motohiro},
   title={Global existence of solutions to semilinear damped wave equation
   with slowly decaying initial data in exterior domain},
   journal={Differential Integral Equations},
   volume={32},
   date={2019},
   pages={615--638},
}


\bib{Sob-2020}{article}{
   author={Sobajima, Motohiro},
   title={Higher order asymptotic expansion of solutions to abstract linear hyperbolic equations},
   journal={Math. Ann.},
   volume={380},
   date={2021},
   pages={1--19},
}

\bib{Tak-2011}{article}{
   author={Takeda, Hiroshi},
   title={Global existence of solutions for higher order semilinear damped
   wave equations},
   journal={Discrete Contin. Dyn. Syst., Dynamical systems, differential equations and applications. 8th
   AIMS Conference. Suppl. Vol. II},
   date={2011},
   pages={1358--1367},
}

\bib{TayYor-2020}{article}{
   author={Taylor, Montgomery},
   author={Todorova, Grozdena},
   title={The diffusion phenomenon for dissipative wave equations in metric
   measure spaces},
   journal={J. Differential Equations},
   volume={269},
   date={2020},
   number={12},
   pages={10792--10838},
}

\bib{TodYor-2001}{article}{
   author={Todorova, Grozdena},
   author={Yordanov, Borislav},
   title={Critical exponent for a semilinear wave equation with damping},
   journal={J. Differential Equations},
   volume={174},
   date={2001},
   number={2},
   pages={464--489},
}

\bib{Zha-2001}{article}{
   author={Zhang, Qi S.},
   title={A blow-up result for a semilinear wave equation with damping: the
   critical case},
   journal={C. R. Acad. Sci. Paris S\'{e}r. I Math.},
   volume={333},
   date={2001},
   number={2},
   pages={109--114},
}

\end{biblist}
\end{bibdiv}
\end{document}